\RequirePackage{fix-cm}
\documentclass[smallcondensed]{svjour3}
\smartqed  %

\usepackage{makeidx}  %

\usepackage{graphicx}
\usepackage{amsfonts, amssymb, amsmath, cite,bm}
\usepackage{algorithm}
\usepackage{algorithmic}
\usepackage{color}
\usepackage{multirow}
\usepackage{multicol}
\usepackage{microtype}
\usepackage[hidelinks,bookmarks=false,hypertexnames=false]{hyperref}
\allowdisplaybreaks
\newtheorem{thm}{Theorem}

\newtheorem{lem}{Lemma}

\newtheorem{defn}{Definition}
  
\newcommand{\eq}{\begin{equation}}
\newcommand{\nq}{\end{equation}}
\newcommand\eqa{\begin{eqnarray}}%
\newcommand\nqa{\end{eqnarray}}

\newcommand\bess{\begin{eqnarray*}}%
\newcommand\eess{\end{eqnarray*}}
\newcommand{\nn}{\nonumber}

\newcommand{\hb}{\hfill $\Box$}

\makeatletter
\def\hlinew#1{%
  \noalign{\ifnum0=`}\fi\hrule \@height #1 \futurelet
   \reserved@a\@xhline}
\makeatother
\makeatletter
\newenvironment{breakablealgorithm}
{%
	\begin{center}
		\refstepcounter{algorithm}%
		\hrule height.8pt depth0pt \kern2pt%
		\renewcommand{\caption}[2][\relax]{%
			{\raggedright\textbf{\ALG@name~\thealgorithm} ##2\par}%
			\ifx\relax##1\relax %
			\addcontentsline{loa}{algorithm}{\protect\numberline{\thealgorithm}##2}%
			\else %
			\addcontentsline{loa}{algorithm}{\protect\numberline{\thealgorithm}##1}%
			\fi
			\kern2pt\hrule\kern2pt
		}
	}{%
		\kern2pt\hrule\relax%
	\end{center}
}
\makeatother
\journalname{}
\renewcommand{\makeheadbox}{}
\begin{document}
\title{The Restricted Inverse Optimal Value Problem on Shortest Path under  $l_\infty$  Norm on Trees
\thanks{ The work of Q.Z. was supported by National Natural Science Foundation of China (Grant
No. 1230012046). The work of X.C. Guan was supported by National Natural Science Foundation of China (62573120). The work of P.M. Pardalos was conducted within the framework of the Basic Research Program at the National Research University Higher School of Economics (HSE).}
}
\titlerunning{Restricted Inverse Optimal Value Problem on Shortest Path }
\author{Qiao Zhang
\and Ying Wang
\and Xiao Li
\and Junhua Jia
\and Xiucui Guan
\and Panos M. Pardalos
}                     %
\authorrunning{Q. Zhang, Y. Wang et al.}%
\institute{Q. Zhang \and Y. Wang  \at Aliyun School of Big Data, Changzhou University, Changzhou 213164, China\\ \email{qiaozhang@cczu.edu.cn} \\X. Li \at Edwardson School of Industrial Engineering, Purdue University, 315 N. Grant Street, West Lafayette, IN 47907-2023, USA\\\email{li5868@purdue.edu} (Corresponding author: Xiao Li)\\
J.H. Jia \at School of Mathematics and Physics, Bengbu University, Bengbu 233030, China\\
X.C. Guan  \at School of Mathematics, Southeast University, Nanjing
210096, China\\ \email{xcguan@163.com} \\
P.M. Pardalos \at Center for Applied Optimization, University of Florida, Gainesville 32611, Florida, USA 	\at
LATNA, Higher School of Economics, Russia
}
\date{}
\maketitle

\begin{abstract}
We study the restricted inverse optimal value problem on shortest paths under the weighted $l_\infty$ norm on trees (RIOVSPT$_\infty$): adjust edge weights at minimum worst-link cost so that the shortest root--leaf path is lower-bounded by a given value $D$, a restriction on a prescribed path $P_0$. Unlike the $l_1$ version minimizing total adjustment cost, the $l_\infty$ objective controls the worst link, the criterion relevant in bandwidth pricing. We characterize feasibility through the minimum cost shortest path interdiction problem (MCSPIT$_\infty$), whose optimal solution $\hat w$ splits the instance at the first edge of $P_0$ into two independent subproblems. Relative to $\hat w$, three cases arise. The first two admit closed-form optimal solutions. The third reduces to a parametric subproblem SPIT$^0_\infty(M)$: minimize the length of $P_0$ on the gate subtree under the budget $C(\hat w)+M$ with all other root--leaf paths kept at least $D$. It is solved exactly by an iterative greedy procedure. On its value function $V(M)$, the predicate $V(M)\leq D$ is equivalent to the nonnegativity of truncated-bound margins, which are monotone, continuous and piecewise affine; the smallest such $M$ gives the optimal value, located by binary search over the edge-bound breakpoints with an exact affine threshold evaluation. The same margin test on the full tree gives a uniform algorithm for all cases. Both algorithms run in $O(n\log n)$ time. Experiments on synthetic trees agree with a linear programming baseline and illustrate the two algorithms' distinct runtime profiles.

\keywords{Inverse optimal value problem; Shortest root--leaf path; Tree; Weighted $l_\infty$ norm; Budget search.}

\end{abstract}

\section{Introduction}

The inverse shortest path problem, denoted by \textbf{(ISP)}, is one of
the earliest inverse problems in network optimization, dating back to
Burton and Toint \cite{BT-MP-ISP-1992}.
Consider a finite weighted graph $G=(V,E,w)$ with node set $V$ and edge
set $E$. Let $P_j$ ($j=0,1,\ldots,r-1$) be all simple paths connecting the
source $s$ and the sink $t$ in $G$ with $P_0$ prescribed. For a vector
$f$ and a path $P_j$, let $f(P_j):=\sum_{e\in P_{j}}f(e)$ for short.
The problem \textbf{(ISP)} seeks a new weight with minimum adjusting cost  
under some measurement such that $P_0$ becomes a shortest $s$--$t$ path
under the new weights.
Then the  problem \textbf{(ISP)} \cite{BT-MP-ISP-1992} admits the following mathematical model.
\begin{eqnarray}
&\min_{\bar w}&\|\bar w-w\| \nonumber\\
\textbf{(ISP)} &s.t.&  \bar w(P_0)\leq  \bar w(P_j),j=1,2,\cdots, r-1,\nn\\
&&\bar w(e)\geq 0, e\in E.\nonumber
\end{eqnarray}

For \textbf{(ISP)} under the $l_2$ norm, Burton and Toint
\cite{BT-MP-ISP-1992} designed a solution algorithm. It builds upon the
Goldfarb--Idnani method for convex quadratic programming and assumes
mutually independent edge costs. They later generalized the algorithm
to correlated edge costs \cite{BT-MP-ISP-1994}. Under the $l_1$ norm,
Zhang et al. \cite{ZM-ZFOR-ISP-1995} formulated \textbf{(ISP)} as a
linear program of special structure and solved it by column generation whose columns are produced by shortest-path algorithms. Xu and Zhang
\cite{XZ-JJIAM-WSP} observed that the feasible adjusted weight vectors
form a polyhedral cone. They described its extreme directions
algebraically and graphically and related the problem to the minimum
cutset problem. Under the Hamming distance, Zhang and Wang
\cite{ZW-JHU-ISPH}  established NP-hardness by a polynomial-time
reduction from the knapsack problem.

A different setting arises in the inverse optimal value problem on shortest path \textbf{(IOVSP)}, where the information given in advance is a target objective value rather than a desired solution. It has drawn less attention  and  comparatively is harder to solve than the problem \textbf{(ISP)}.  Formally, on  the weighted graph $G=(V,E,w)$,
 the problem \textbf{(IOVSP)} aims at adjusting the weight vector to minimize the total adjustment cost under some measurement such that the length of the shortest path equals a prescribed value $D$. The mathematical model of the problem \textbf{(IOVSP)} \cite{CH-N-ISPL} can be stated as follows. 
\begin{eqnarray}
&\min_{\bar w}&\|\bar w-w\| \nonumber\\
\textbf{(IOVSP)} &s.t.& \min_{0\leq j\leq r-1} \bar w(P_j)=  D, \nn\\
&&\bar w(e)\geq 0, e\in E.\nonumber
\end{eqnarray}
Under the weighted $l_1$ norm, Cui and Hochbaum \cite{CH-N-ISPL}
generalized \textbf{(IOVSP)} to the inverse shortest path lengths problem
\textbf{(ISPL)}, which imposes target distances on several source--sink
paths simultaneously. They established its NP-hardness on general
graphs, and identified polynomially solvable cases through a lower
bound relaxation.For linear programming, Ahmed and Guan \cite{AG-MP-IOVP} introduced the
inverse optimal value problem \textbf{(IOVLP)}, which seeks a cost
vector making the optimal value of the resulting \textbf{(LP)} as close
as possible to a desired value, and proved it NP-hard. Lv et al.
\cite{Lv-JCAM-IOVP} reformulated \textbf{(IOVLP)} as a bilevel program
and converted it, via the Kuhn--Tucker conditions, into an ordinary
nonlinear program, for which they developed an exact penalty method with
convergence analysis.

Based on the bandwidth pricing problem of Paleologo and Takriti
\cite{PT-ANV-BT}, Zhang et al. \cite{ZhangQ22} introduced a restricted
version of \textbf{(IOVSP)}: the restricted inverse optimal value problem
on shortest path on trees, denoted by \textbf{(RIOVSPT)}. Let $T=(V, E, w)$ be an edge-weighted tree rooted at $v_1$, where $V:=\{v_1, v_2, \cdots, v_n\}$ and  $E:=\{e_2, e_3, \cdots, e_{n}\}$ are the sets of nodes and edges, respectively. Let $Y:=\{t_0,t_1,\ldots,t_{r-1}\}$ be the set of rooted leaves. Let $w(e)$, $l(e)$ and $u(e)$ be the  original weight, the lower and upper bounds of the adjusted weight for an edge $e\in E$, respectively,  where $l(e),w(e),u(e)$ are finite real numbers satisfying $l(e)\leq w(e)\leq u(e)$.  %
Let $c(e)>0$ be the cost of adjusting the weight of edge $e\in E$. Denoted by $P_{v_i, v_j}$ the unique path from $v_i$ to $v_j$ on the tree $T$ and record $P_k:=P_{v_1, t_k}$ as the ``root-leaf path" of the leaf node $t_k(k=0,1,2,\cdots,r-1)$ in brief. 
The problem \textbf{(RIOVSPT)} aims at adjusting the weights of some edges to minimize the total adjustment cost under some measurement such that the length of the shortest root-leaf path of the tree is lower-bounded by a given value $D$, while the length of the given path $P_0$ is required to be exactly $D$. The mathematical model can be stated below.
\begin{eqnarray}
&\min_{\bar w}&\|\bar w-w\| \nonumber\\
\textbf{ (RIOVSPT) }&s.t.&\bar w(P_k)\geq D, t_k\in Y\setminus \{t_0\}, \nonumber\\
&&  \bar w(P_0)=D,\nonumber\\
&& l(e)\leq \bar w(e)\leq u(e), e\in E.\nonumber
\end{eqnarray}
Notice that the feasible region of \textbf{(RIOVSPT)} can be regarded as the intersection of those of \textbf{(ISP)} and \textbf{(IOVSP)} restricted on trees.

 Under the weighted $l_1$ norm, Zhang et al. \cite{ZhangQ22} designed an $O(n^2)$ time algorithm for the problem \textbf{(RIOVSPT)} on trees, together with a linear time algorithm under the unit $l_1$ norm.
Similar to the definition of the problem \textbf{(RIOVSPT)},
Zhang et al. \cite{ZGP-LBIOVMST-JGO,ZGZ-IOVMST-OL} proposed the inverse optimal value problem on minimum spanning tree \textbf{(RIOVMST)}. It requires not only that a given spanning tree $T^0$ becomes a minimum spanning tree, but also that its weight equals a prescribed value $K$. They measured the adjustment cost under the unit $l_{\infty}$ norm and developed two strongly polynomial time algorithms with running time $O(|V||E|)$ for the unbounded case \cite{ZGZ-IOVMST-OL} and the lower bounded case \cite{ZGP-LBIOVMST-JGO}, respectively.
Additionally, Wang et al. \cite{WGZ-CIOVMST-JOCO} proposed an $O(|V||E| \log |E|)$ time algorithm for the bounded problem  by measuring the adjustment cost under the bottleneck Hamming distance. More recently, Jia et al. \cite{Jia2023}
studied the restricted bounded problem under the weighted $l_{\infty}$
norm. Their strongly polynomial combinatorial algorithm runs in
$O(|E|^{2}|V|)$ time, which they improved to $O(|E|^{2}\log|V|)$ by a
binary search method. Under the weighted $l_{1}$ norm, Jia et al. \cite{Jia2024} formulated
the problem as a linear program and solved it combinatorially, via LP
duality and a binary search over a critical value.   They \cite{Jia2026} further studied the restricted inverse optimal value problem on linear programming under the weighted $l_1$ norm.

Zhang et al. \cite{ZhangQ20,ZhangQ21} considered the minimum cost shortest path interdiction problem by upgrading edges on trees \textbf{(MCSPIT)}, in which the restriction on the path $P_0$ in the problem \textbf{(RIOVSPT)} is dropped. The problem \textbf{(MCSPIT)} aims at minimizing the total cost of upgrading some edges under  some measurement such that the length of the shortest path of the tree is lower-bounded by a given value $D$.  The mathematical model can be stated as follows.
\begin{eqnarray}
&\min_{\tilde w}& \|\tilde w-w\| \nonumber\\
\textbf{ (MCSPIT) }&s.t.&\tilde w(P_k)\geq D, t_k\in Y, \nonumber\\
&& w(e)\leq \tilde w(e)\leq u(e), e\in E.\nonumber
\end{eqnarray}

Under the weighted $l_1$ norm, Zhang et al. \cite{ZhangQ20} gave an
$O(n^2)$ primal-dual algorithm for \textbf{(MCSPIT$\bm{_1}$)}  and a
linear-time algorithm under the unit $l_1$ norm,respectively. Under the Hamming
distance, they \cite{ZhangQ21} proved the weighted version
NP-hard and gave an $O(n^4\log n)$ algorithm for the unit case  based on
dynamic programming algorithm. Subsequently, Lei et al. \cite{YiL22} 
improved this bound to $O(n^3\log n)$. Under the weighted $l_\infty$
norm, Zhang and Li \cite{ZhangL2025} designed a $O(n\log n)$ algorithm
for \textbf{(MCSPIT$\bm{_\infty}$)}. Its optimal solution serves as the
initial solution of this paper.

Interdiction problems on trees with respect to the sum of root-leaf distances have also attracted much attention recently. Zhang et al. \cite{ZGJQ22} studied the problem   by upgrading edges
or nodes on trees, giving $O(n)$ algorithms under the weighted $l_1$
norm and proving NP-hardness under the weighted Hamming distance. Li et al. \cite{Li2024} further imposed a cardinality constraint on the
upgraded edges and gave $O(n\log n)$ binary-search algorithms. Li et
al. \cite{Li2025} studied a double interdiction version with both a
budget constraint and a lower bound on the shortest root--leaf distance,
proving it NP-hard in general and giving a linear-time algorithm for the
single-upgrade case. A systematic summary of this line can be found in
the recent monograph chapter \cite{Xiao2025}.

When the weighted $l_\infty$ norm is adopted as the measurement $\|\cdot\|$, the problems \textbf{(RIOVSPT$\bm{_\infty}$)} and \textbf{(MCSPIT$\bm{_\infty}$)} can be formulated in the following forms, respectively.
\begin{eqnarray}
&\min_{\tilde w}& C(\tilde w):=\max_{e\in E}c(e)|\tilde w(e)-w(e)| \nonumber\\
\textbf{RIOVSPT$\bm{_\infty}$}\bm{(T)} &s.t.&\tilde w(P_k)\geq D, t_k\in Y\setminus\{t_0\}, \nonumber\\
&& \tilde w(P_0)=D,\label{eq-RIOVSPTinfty}\\
&& l(e)\leq \tilde w(e)\leq u(e), e\in E.\nonumber
\end{eqnarray}
\begin{eqnarray}
&\min_{\hat w}& C(\hat w):=\max_{e\in E}c(e)|\hat w(e)-w(e)| \nonumber\\
\textbf{MCSPIT$\bm{_\infty}$}\bm{(T,w,c,u,D)}&s.t.&\hat w(P_k)\geq D, t_k\in Y, \label{eq-MCSPITinfty}\\
&& w(e)\leq \hat w(e)\leq u(e), e\in E.\nonumber
\end{eqnarray}

The weighted $l_\infty$ measurement adopted in the problem \textbf{(RIOVSPT$\bm{_\infty}$)} is also well motivated in the  bandwidth pricing scenario in Fig.\ref{Fig-Background}\cite{ZhangQ22}. In practice, repricing a link is never free. Changing the quoted price
of a link $e$ by one unit incurs a penalty $c(e)$. This penalty
reflects the sensitivity of the existing contracts through that link.
It also captures the risk of customer churn and the attention from
regulators. Minimizing the total cost may then pile a drastic change
onto a few links. Such a change forces the renegotiation of the affected contracts,
or even violates service-level agreements. What the provider really
cares about is therefore the worst-affected link rather than the total
bill. This makes the $l_\infty$ objective
$C(\tilde w):=\max_{e\in E}c(e)|\tilde w(e)-w(e)|$ the natural choice.
It controls the worst weighted adjustment directly and guarantees that
no single link is repriced beyond tolerance. It need not equalize the
adjustments, nor does it impose an externally specified threshold unless
one is added separately.
The pricing requirements are unchanged. Every side route must be priced at least $D$. A cheaper one would
undercut the reference price and open an arbitrage. The designated
route $P_0$ is reserved for emergency communication and must be priced
at exactly $D$. This keeps it as cheap as any alternative, so that it
remains usable when an emergency occurs. 
Under the worst-link objective, these requirements no longer decouple
from the cost. Adjusting the edges of $P_0$ toward this target directly moves the
worst weighted deviation on the shared links.  Resolving this coupling is precisely the task of
\textbf{(RIOVSPT$\bm{_\infty}$)}.

\begin{figure}[!htbp]
	\centering
	\includegraphics[totalheight=2in]{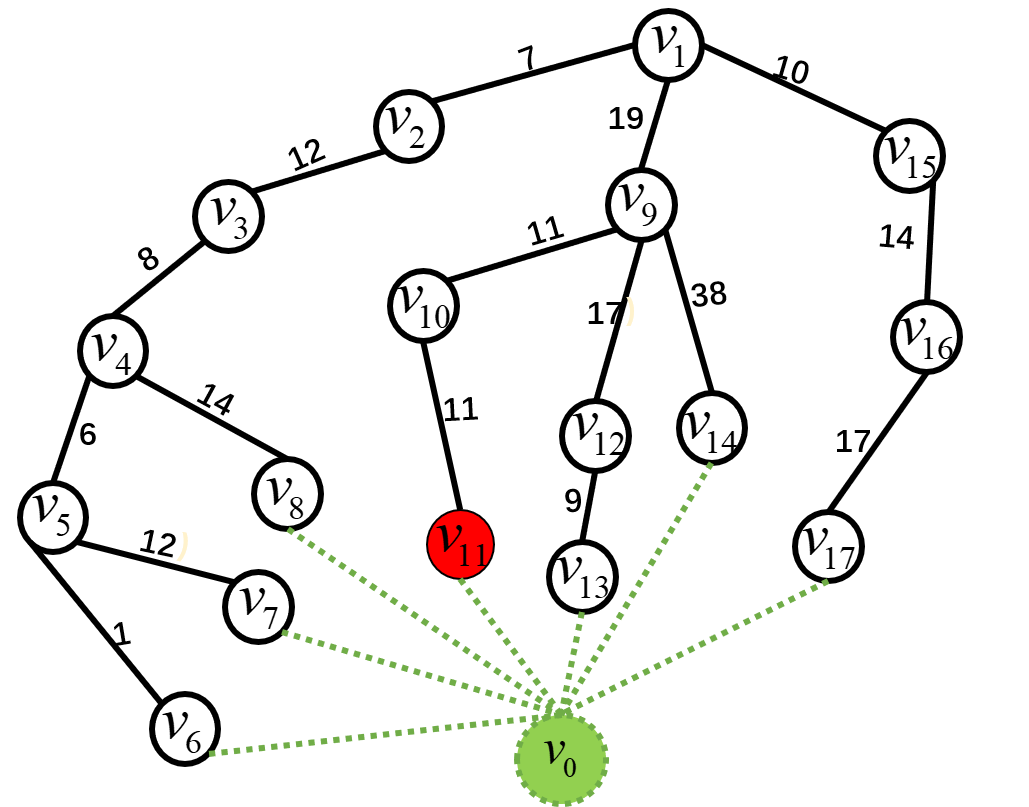}
	\caption{A special bandwidth
network between the cities $v_1$ and
$v_0$ with prices on the links.  The route through $v_{11}$ is reserved for emergency communication.} \label{Fig-Background}
\end{figure}

Zhang et al.~\cite{ZhangQ22}  solved the problem \textbf{(RIOVSPT)} under the
weighted $l_1$ norm  with  an $O(n^2)$ algorithm based on minimum cost
cuts. This approach suits the aggregate cost criterion. But the total
cost does not control the largest single-edge deviation: a drastic
change may hide in one link.  We therefore study the weighted $l_\infty$ objective, which
controls the maximum weighted adjustment while the bounds and path
constraints stay unchanged.

The change of objective is more than a modeling nuance. 
On the modeling side, a fixed maximum-adjustment budget truncates the
feasible interval of every edge.  It is
not a separable per-edge cost. On the algorithmic side,
the $l_1$ method works through minimum-cost cuts precisely because a
total cost decomposes edge by edge, whereas the worst-case objective
does not. 
New machinery is therefore necessary.

In this paper, we solve the problem exactly in $O(n\log n)$ time. We
analyze the structural properties of the problem, reduce the hard
case to a parametric subproblem, and locate the optimal value by a
binary search over the edge-bound breakpoints. The same machinery
applies verbatim to the full tree, yielding a uniform algorithm with
the same complexity.

The rest of the paper is organized as follows. Section~2 presents the details of the solution in two ways. Section~2.1 introduces the notation, recalls the
feasibility conditions, and establishes the decomposition used in the
case analysis. Sections~2.2--2.4 derive the optimal values and
solution constructions for the three cases relative to an auxiliary
interdiction solution; the budget characterization is developed in
Section~2.4. Section~\ref{subsec-uniform} applies this
characterization directly to the full tree and gives the uniform
algorithm. Section~\ref{sec:computational-study} reports
numerical experiments on correctness, running time, and case-specific
behavior. Section~\ref{sec:conclusion} concludes the paper.

\section{Solving the problem (RIOVSPT$_{\infty}$)}
In this section, we develop an exact combinatorial algorithm for the problem \textbf{(RIOVSPT$\bm{_\infty}$)}. We first introduce some necessary definitions and notation, most of which are inherited from \cite{ZhangQ20,ZhangQ21}. Based on these, we  characterize the feasibility of the problem. Then, starting from an optimal solution of the problem \textbf{(MCSPIT$\bm{_\infty}$)}, we show that the problem can be decomposed at the first edge of $P_0$ into two independent subproblems, so that it suffices to solve the subproblem on the part containing $P_0$.

\subsection{Definitions and feasibility}
We first recall the concepts of the layer of a node and the layer number of an edge introduced in \cite{ZhangQ20}, which decompose the tree into edge-disjoint chains.

\begin{defn}\cite{ZhangQ20}\label{defn-tab}
	For $e_j=(v_i, v_j)$, we call $v_i$  the \textbf{father} of $v_j$, denoted by $father(v_j)=v_i$. Define $Layer(v_1):=1$ and the \textbf{Layer} of any other nodes $v\in {V\setminus \{v_1\}}$ as
	$$Layer(v):=\left\{
	\begin{array}{lr}
	Layer(father(v)), & \text{if }\ deg(v)\leq{2},\\
	Layer(father(v))+1, & \text{if }\ deg(v)>2.
	\end{array} \right.$$
\end{defn}
\begin{defn}\cite{ZhangQ20}\label{defn-LN-edge}
	For each edge $e_j=(v_i,v_j)\in E$ with $Layer(v_i)\leq Layer(v_j),$ we define $LN(e_j):=Layer(v_i)$ as the \textbf{layer number} of the edge $e_j$.
\end{defn}
Denote by $V^*:=\{v\in V\setminus\{v_1\}\mid deg(v)>2\}$ the set of branching nodes of $T$.
Based on the node layers, we next define the critical descendants of a
branching node. They serve as the endpoints of the chains that
decompose the tree.

\begin{defn}\cite{ZhangQ20}
	For a  node $\bar v\in V^*\cup\{v_1\}$, we define a set $CD(\bar v)$ of \textbf{critical descendants}.
	Let $\bar v$ be on the path from $v_1$ to $v\in V^*\setminus\{v_1\}\cup Y$. If $v\in V^*\setminus\{v_1\}$ and $Layer(v)=Layer(\bar v)+1$, then $v\in{CD(\bar v)}$; if $v\in Y$ and $Layer(v)=Layer(\bar v)$, then $v\in{CD(\bar v)}$. Correspondingly, if $v\in CD(\bar v)$,  we call $\bar v$ the \textbf{critical ancestor} of $v$, denoted by $CA(v):=\bar v$.
\end{defn}
\begin{figure}
    \centering
\includegraphics[width=0.5\linewidth]{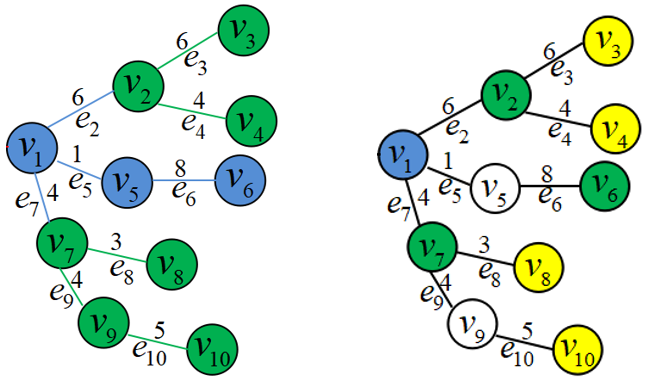}
    \caption{A rooted tree with adjustment coefficients $c(e)$ on the edges. Left: blue and green nodes have layers 1 and 2, respectively. Right: the green nodes are the critical descendant of $v_1$.}
    \label{LNCCV}
\end{figure}

For instance, in Fig.\ref{LNCCV}, $CD(v_1):=\{v_2, v_6, v_7\}$ and $CA(v_2):=v_1$.

With the critical ancestors defined above, the tree can be decomposed into chains as follows.

\begin{defn}\cite{ZhangQ20}\label{defn-chain}
	For any  node $v\in V^*\backslash\{v_1\}\cup Y$, define $\chi_v:=P_{CA(v),v}$ as the chain from $CA(v)$ to $v$. %
\end{defn}

Before constructing an optimal solution, we first characterize the feasibility of the problem \textbf{(RIOVSPT$\bm{_\infty}$)}. Notice that feasibility is independent of the measurement $\|\cdot\|$.
It concerns only the existence of a weight vector within the bounds
that meets the path-length requirements. Hence the characterization,
and its proof  are identical to those for the $l_1$ norm
(\cite{ZhangQ22}, Theorem~2).

\begin{thm}
\label{thm-fea=RIOVSPT_infty}
	The problem \textbf{ (RIOVSPT$\bm{_{\infty}}$)} %
	is feasible if and only if \textbf{(1)} $l(P_0)\leq D\leq \min_{t_k\in Y}u(P_k)$ and \textbf{(2)} $l(P_0 )\leq \min_{t_k\in Y}(l(P_k\cap P_0)+u(P_k\backslash P_0))$ hold.	
\end{thm}

\begin{corollary}\label{cor-fea=RIOVSPT1}
	If $D<l(P_0)$ or $ D>\min_{t_k\in Y}u(P_k)$ or $l(P_0)>\min_{t_k\in Y}(l(P_k\cap P_0)+u(P_k\backslash P_0))$, then the problem \textbf{ (RIOVSPT$\bm{_{\infty}}$)} is infeasible.
\end{corollary}

Raising the downgraded edges of a feasible solution of
\textbf{(RIOVSPT$\bm{_\infty}$)} back to the original weights yields a
feasible solution of \textbf{(MCSPIT$\bm{_\infty}$)} at no larger cost.
The following lemma records it, which will be invoked
repeatedly in the subsequent proofs.
\begin{lem}\label{lem-lift}
Let $\tilde w$ be a feasible solution of
\textbf{(RIOVSPT$\bm{_\infty}$)}. Then the vector
$\tilde w'(e):=\max\{\tilde w(e),w(e)\}(e\in E)$ is feasible for
\textbf{(MCSPIT$\bm{_\infty}$)} and satisfies
$C(\tilde w')\le C(\tilde w)$.
\end{lem}
\begin{proof}
By definition, we have $w(e)\le\tilde w'(e)\le u(e)$ for $e\in E$. Additionally, 
$\tilde w'(P_k)\ge\tilde w(P_k)\ge D$ for every
$t_k\in Y\setminus\{t_0\}$. What's more, we have 
$\tilde w'(P_0)\ge\tilde w(P_0)=D$. Hence $\tilde w'$ is feasible for
\textbf{(MCSPIT$\bm{_\infty}$)}. Moreover, for $e\in E$, we have
\[
c(e)\bigl|\tilde w'(e)-w(e)\bigr|
=c(e)\max\{\tilde w(e)-w(e),0\}
\le c(e)\bigl|\tilde w(e)-w(e)\bigr|.
\]
Hence,  $C(\tilde w')\le C(\tilde w)$.\hb
\end{proof}

An immediate consequence is a universal lower bound: every feasible
solution of \textbf{(RIOVSPT$\bm{_\infty}$)} never costs
less than the problem \textbf{(MCSPIT$\bm{_\infty}$)}.
\begin{lem}\label{lem-lowbd-hard}
Let $\tilde w$ be a feasible solution of
\textbf{(RIOVSPT$\bm{_\infty}$)}, then $C(\tilde w)\geq C(\hat w)$.
\end{lem}
\begin{proof}
By Lemma \ref{lem-lift}, the vector  $\tilde w'(e):=\max\{\tilde w(e),w(e)\}(e\in E)$ is feasible for
\textbf{(MCSPIT$\bm{_\infty}$)} and satisfies
$C(\tilde w')\le C(\tilde w)$. 
Since $\hat w$ is an optimal solution of
\textbf{(MCSPIT$\bm{_\infty}$)}, we have
$C(\hat w)\le C(\tilde w')\le C(\tilde w)$. Then $C(\tilde w)\ge
C(\hat w)$ holds for every feasible solution $\tilde w$ of
\textbf{(RIOVSPT$\bm{_\infty}$)}.\hb
\end{proof}

We next construct  an optimal solution. The starting point is
the problem \textbf{(MCSPIT$\bm{_\infty}$)}, in which the restriction on
$P_0$ is dropped. An optimal solution $\hat w$ of \textbf{MCSPIT$\bm{_\infty}$}$\bm{(T)}$
can be computed in $O(n\log n)$ time by the algorithm of Zhang and Li
\cite{ZhangL2025}. It satisfies $\hat w(P_k)\ge D$ for every
$t_k\in Y$.  In particular, $\hat w(P_0)\ge D$. A feasible solution of
\textbf{(RIOVSPT$\bm{_\infty}$)} in model (\ref{eq-RIOVSPTinfty}) must
meet more: $\tilde w(P_k)\ge D$ for every $t_k\in Y\setminus\{t_0\}$
and $\tilde w(P_0)=D$ simultaneously. If $\hat w(P_0)=D$, then
$\hat w$ already meets all the requirements and the following theorem shows that it is indeed optimal.
\begin{thm}\label{thm-P0D}
	Let $\hat w$ be an optimal solution of the problem \textbf{(MCSPIT}$\bm{_{\infty})}$. If $\hat w(P_0)=D$, then $\tilde w:=\hat w$ is also an optimal solution of the problem \textbf{(RIOVSPT$\bm{_{\infty}}$)}.
\end{thm}
\begin{proof}
Feasibility is obvious. It remains to show the optimality of $\hat w$.
Suppose  there exists a feasible solution $\bar w$
of the problem \textbf{(RIOVSPT$\bm{_\infty}$)} with
$C(\bar w)<C(\hat w)$. By
Lemma~\ref{lem-lift}, the vector $w':=\max\{\bar w,w\}$ is feasible
for \textbf{(MCSPIT$\bm{_\infty}$)} and satisfies $C(w')\le C(\bar w)$.
Hence, $C(w')\leq C(\bar w)<C(\hat w)$, contradicting the
optimality of $\hat w$ for \textbf{(MCSPIT$\bm{_\infty}$)}. Therefore
$\tilde w:=\hat w$ is an optimal solution of
\textbf{(RIOVSPT$\bm{_\infty}$)}.\hb
\end{proof}
It remains to consider the case $\hat w(P_0)>D$, in which additional adjustment on the path $P_0$ is required. We show that in this case the problem can be decomposed at the first edge of $P_0$ into two independent subproblems.
For any edge $e_j=(v_i,v_j) \in E$, define $T_{e_j}:=P_{v_1,v_j}\cup T_{v_j}$, where $T_{v_j}$ is the subtree of $T$ rooted at $v_j$. Let $\ddot{e}:=(v_1,\ddot v)\in P_0$ be the \textbf{gate} edge, which is the first edge from the root $v_1$ on the path $P_0$, and denote by $\ddot T:=T_{\ddot e}=\{\ddot e\}\cup T_{\ddot v}$ in brief. Then the tree $T=(T\backslash \ddot T)\cup\ddot T$ is divided into two parts, and every root-leaf path of $T$ lies entirely in one of them: the path $P_0$, together with all the root-leaf paths ending in $\ddot T$, lies in $\ddot T$, while the remaining root-leaf paths lie in $T\backslash\ddot T$. Consequently, the problem \textbf{RIOVSPT$\bm{_\infty}$}$\bm{(T)}$ is decomposed into two independent subproblems \textbf{RIOVSPT$\bm{_\infty}$}$\bm{(T\backslash \ddot T)}$ and \textbf{RIOVSPT$\bm{_\infty}$}$\bm{(\ddot T)}$. Since $P_0$ is contained in $\ddot T$, the former subproblem involves only the lower-bound constraints on the path lengths. Its optimal solution can be obtained from an optimal solution of the problem \textbf{MCSPIT$\bm{_\infty}$}$\bm{(T\backslash \ddot T)}$: replacing a weight below $w(e)$ by $w(e)$ can only increase path lengths and cannot increase the adjustment objective, so that the two subproblems on $T\backslash\ddot T$ share the same optimal value.

\begin{thm}\label{thm-construct}
	Let $\ddot{ w}$ be an optimal solution of the subproblem \textbf{RIOVSPT$\bm{_\infty}$}$\bm{(\ddot T)}$ and $\hat w$ be an optimal solution of the problem \textbf{MCSPIT$\bm{_\infty}$}$\bm{(T\backslash \ddot T)}$. Then
	\begin{equation}
		\tilde w(e):=\left\{\begin{array}{ll}
			\hat w(e),&e\in T\backslash \ddot T,\\
			\ddot{ w}(e), &e\in \ddot T
		\end{array}\right.\label{eq-construct0}
	\end{equation}
	is an optimal solution of the problem \textbf{RIOVSPT$\bm{_\infty}$}$\bm{(T)}$, and the relevant optimal objective value is $C(\tilde w)=\max\{C(\hat w), C(\ddot w)\}$.
\end{thm}
\begin{proof}
	The two parts are edge-disjoint: $E(\ddot T)\cap E(T\setminus\ddot
T)=\emptyset$. Moreover, every root--leaf path of $T$ lies entirely in
one of them. Hence the constraints on $\ddot T$ and those on
$T\setminus\ddot T$ are independent of each other, and the objective is
separable with respect to the two parts:
	\begin{eqnarray*}
	    C(\tilde w)&=&\max_{e\in E}c(e)|\tilde w(e)-w(e)|\\&=&\max\left\{\max_{e\in T\backslash\ddot T}c(e)|\tilde w(e)-w(e)|,\ \max_{e\in \ddot T}c(e)|\tilde w(e)-w(e)|\right\}.
	\end{eqnarray*}
	We first show the feasibility of the solution $\tilde w$ in (\ref{eq-construct0}). The bounds $l(e)\leq\tilde w(e)\leq u(e)$ hold piecewise by the feasibility of $\hat w$ and $\ddot w$. For a leaf $t_k$ contained in $T\backslash\ddot T$, we have $\tilde w(P_k)=\hat w(P_k)\geq D$. For a leaf $t_k$ contained in $\ddot T$ with $t_k\neq t_0$, we have $\tilde w(P_k)=\ddot w(P_k)\geq D$. Finally, $\tilde w(P_0)=\ddot w(P_0)=D$ as $P_0\subseteq \ddot T$. Hence, $\tilde w$ is feasible. Moreover, $C(\tilde w)=\max\{C(\hat w),C(\ddot w)\}$ follows directly from the separability.
	
	We next show the optimality of $\tilde w$. Let $\tilde w'$ be an
arbitrary feasible solution of \textbf{RIOVSPT$\bm{_\infty}$}$\bm{(T)}$.
Its restrictions to $T\setminus\ddot T$ and to $\ddot T$ are feasible
for the two subproblems, respectively. Consider the restriction to
$T\setminus\ddot T$. As in the proof of Theorem~\ref{thm-P0D}, raising
its downgraded edges back to the original weights gives a feasible
solution of \textbf{MCSPIT$\bm{_\infty}$}$\bm{(T\setminus\ddot T)}$ with
no larger objective value. Hence the restriction of $\tilde w'$ to
$T\setminus\ddot T$ has objective value at least $C(\hat w)$. Therefore,
	$$C(\tilde w')=\max\left\{C_{T\backslash\ddot T}(\tilde w'),\, C_{\ddot T}(\tilde w')\right\}\geq \max\{C(\hat w),C(\ddot w)\}=C(\tilde w),$$
	where $C_{T\backslash\ddot T}(\cdot)$ and $C_{\ddot T}(\cdot)$ denote the objective values restricted to the two parts. Hence, $\tilde w$ is an optimal solution of the problem \textbf{RIOVSPT$\bm{_\infty}$}$\bm{(T)}$ with the optimal value $C(\tilde w)=\max\{C(\hat w), C(\ddot w)\}$. \hb
\end{proof}

By Theorem \ref{thm-construct}, it remains to solve the subproblem \textbf{RIOVSPT$\bm{_\infty}$}$\bm{(\ddot T)}$, on which we focus in the remainder of this paper. Denote by $\ddot Y:=\{t_k\in Y\mid t_k\in \ddot T\}$ the set of leaves contained in $\ddot T$, and recall that $\chi_{t_0}$ is the chain ending at the leaf $t_0$. 
Let $\hat w$ with cost $C(\hat w)$  be an optimal solution of the
problem \textbf{MCSPIT$\bm{_\infty}$}$\bm{(\ddot T\setminus\chi_{t_0})}$,  which can be  computed by Algorithm~1
of \cite{ZhangL2025}. 
The optimal solution returned by Algorithm~1 of \cite{ZhangL2025} has
the level-cap form
\[
\hat w(e):=\min\left\{u(e),\,w(e)+\frac{C(\hat w)}{c(e)}\right\},
\qquad e\in\ddot T\setminus\chi_{t_0},
\]
which we record as the \textbf{canonical optimal solution} of
\textbf{MCSPIT$\bm{_\infty}$}$\bm{(\ddot T\setminus\chi_{t_0})}$. 

Extend the vector by $\hat w(e):=w(e)$ on $\chi_{t_0}$, which leaves the optimal value unchanged. 
For convenience, we introduce the following
notation. For each $e\in E$, let
\begin{align*}
\Delta l(e) &:= l(e)-w(e)\le 0,\qquad
 \Delta u(e) := u(e)-w(e)\ge 0,\\
\Delta\hat l(e) &:= -\min\Bigl\{\frac{C(\hat w)}{c(e)},\,-\Delta l(e)\Bigr\}
                  = \max\Bigl\{-\frac{C(\hat w)}{c(e)},\,\Delta l(e)\Bigr\},
\\ \Delta\hat u(e) &:= \min\Bigl\{\frac{C(\hat w)}{c(e)},\,\Delta u(e)\Bigr\}.
\end{align*}

Moreover, recall that for a vector $f$ and a path $P_j$, we let $f(P_j):=\sum_{e\in P_{j}}f(e)$ for short. Then for  the chain $\chi_{t_0}$, we have 
$
\Delta\hat l(\chi_{t_0}):=\sum_{e\in\chi_{t_0}}\Delta\hat l(e)
$ and $
\Delta\hat u(\chi_{t_0}):=\sum_{e\in\chi_{t_0}}\Delta\hat u(e).
$

The deviation $D^{0}:=D-\hat w(P_{0})$ measures the gap between the
target $D$ and the current length of $P_{0}$. Together with the
adjustment range of the  chain $\chi_{t_0}$, it distinguishes
the three cases below.

\subsection{The case when $\Delta \hat l(\chi_{t_0})\leq D^0\leq \Delta \hat u(\chi_{t_0})$. }
If $\Delta\hat l(\chi_{t_0})\le D^0\le\Delta\hat u(\chi_{t_0})$, then an
optimal solution of \textbf{RIOVSPT$\bm{_\infty(\ddot T)}$} can be obtained
by a direct adjustment of the terminal chain.

\begin{thm}\label{thm-case-mid}
When $\Delta \hat l(\chi_{t_0})\leq D^0\leq \Delta \hat u(\chi_{t_0})$ in the  problem  \textbf{RIOVSPT$\bm{_{\infty}(\ddot T)}$},

(1.1) if $\Delta \hat l(\chi_{t_0})\leq D^0\leq 0,$   determine the edge $e_b=(v_a,v_b)\in \chi_{t_0}$ satisfying \begin{eqnarray}\Delta \hat l(P_{CA(t_0),v_b})\leq D^0\leq \Delta \hat l(P_{CA(t_0),v_a}),\label{eq-th-case1_1}\end{eqnarray}  then 
\begin{eqnarray}\tilde w(e)=\left\{\begin{array}{ll}w(e)+\Delta \hat l(e),&e\in P_{CA(t_0),v_a}\\ w(e)+D^0-\Delta \hat l(P_{CA(t_0),v_a}),& e=e_b\\\hat w(e),&\text{otherwise}\end{array}\right.\label{tilde w-case-1}\end{eqnarray} is an optimal solution of the problem \textbf{RIOVSPT$\bm{_{\infty}(\ddot T)}$}.

(1.2) if $0< D^0\leq \Delta \hat u(\chi_{t_0}),$  determine the edge $e_b=(v_a,v_b)\in \chi_{t_0}$ satisfying\begin{eqnarray}\Delta \hat u( P_{CA(t_0),v_a}) \leq D^0\leq\Delta \hat u(P_{CA(t_0),v_b}),\label{eq-th-case1_2}\end{eqnarray}
 then \begin{eqnarray}\tilde w(e)=\left\{\begin{array}{ll}w(e)+\Delta \hat u(e),&e\in P_{CA(t_0),v_a}\\w(e)+D^0-\Delta \hat u(P_{CA(t_0),v_a}),& e=e_b\\ \hat w(e),&\text{otherwise}\end{array}\right.\label{tilde w-case-2}\end{eqnarray} 
is an optimal solution of the problem \textbf{RIOVSPT$\bm{_{\infty}(\ddot T)}$}.

\end{thm}
\begin{proof}
	(1.1) We first show the feasibility of the solution $\tilde w$.
 \textbf{About the bounds,} we focus on checking $e_b$. By
(\ref{eq-th-case1_1}) and $\Delta \hat l(e_b)\geq \Delta l(e_b)=l(e_b)-w(e_b)$ , we have 
$D^0\ge\Delta\hat l(P_{CA(t_0),v_b})
=\Delta\hat l(e_b)+\Delta\hat l(P_{CA(t_0),v_a})
\ge l(e_b)-w(e_b)+\Delta\hat l(P_{CA(t_0),v_a}).$
 Together with $D^0-\Delta\hat l(P_{CA(t_0),v_a})\leq 0$ from (\ref{eq-th-case1_1}), we have 
\begin{eqnarray*}
    l(e_b)\le w(e_b)+D^0-\Delta\hat l(P_{CA(t_0),v_a})
=\tilde w(e_b)\le w(e_b)\le u(e_b).\label{th4-eq-bound-eb-1}
\end{eqnarray*}
\textbf{About the path-length constraints,} we have
 \begin{eqnarray*}\tilde w(P_0)&=&\hat w(P_{v_1,CA(t_0)})+w(P_{CA(t_0),v_a})+\Delta \hat l(P_{CA(t_0),v_a})\\&&+w(e_b)+
  	D^0-\Delta \hat l(P_{CA(t_0),v_a})+w(P_{v_b,t_0})\\&=&\hat w(P_0)+\Delta \hat l(P_{CA(t_0),v_a})+D^0-\Delta \hat l(P_{CA(t_0),v_a})=D.
 \end{eqnarray*} For other $t_k\in \ddot Y\backslash\{t_0\}$,  we have $\tilde w(P_k)=\hat w(P_k)\geq D $. Hence, the feasibility holds. 

 We next verify the cost and the optimality. For $e\in P_{CA(t_0),v_a}$,
	\begin{eqnarray*}
	c(e)\big|\tilde w(e)-w(e)\big|=-c(e)\Delta \hat l(e)=c(e)\min\Big\{\frac{C(\hat w)}{c(e)},\,-\Delta l(e)\Big\}\leq C(\hat w).
	\end{eqnarray*}
	For $e_b$, since $0\leq w(e_b)-\tilde w(e_b)=\Delta \hat l(P_{CA(t_0),v_a})-D^0\leq -\Delta \hat l(e_b)$ by  (\ref{eq-th-case1_1}), we have
	$$
	c(e_b)\big|\tilde w(e_b)-w(e_b)\big|\leq -c(e_b)\Delta \hat l(e_b)\leq C(\hat w). 
	$$
    For other edges,  	we have  $c(e)\big|\tilde w(e)-w(e)\big|=c(e)\big|\hat w(e)-w(e)\big|\leq C(\hat w)$. Therefore, $C(\tilde w)\leq C(\hat w)$. By Lemma \ref{lem-lowbd-hard}, we have 
	$C(\tilde w)\geq C(\hat w)$. Then $C(\tilde w)= C(\hat w)$ and $\tilde w$ is an optimal solution of the problem \textbf{RIOVSPT$\bm{_{\infty}(\ddot T)}$}.

	(1.2)  We first show the feasibility of the solution $\tilde w$.
\textbf{About the bounds,} we focus on checking $e_b$. By (\ref{eq-th-case1_2}) and $\Delta \hat u(e_b)\leq \Delta u(e_b)=u(e_b)-w(e_b)$, we have
	$
		D^0\leq \Delta \hat u(P_{CA(t_0),v_b})=\Delta \hat u(P_{CA(t_0),v_a})+\Delta \hat u(e_b)\leq \Delta \hat u(P_{CA(t_0),v_a})+u(e_b)-w(e_b).
		$\\
	Together with  $D^0-\Delta \hat u(P_{CA(t_0),v_a})\geq 0$ from (\ref{eq-th-case1_2}), we have  \begin{eqnarray*}
	    l(e_b)\leq w(e_b)\leq \tilde w(e_b)=w(e_b)+D^0-\Delta \hat u(P_{CA(t_0),v_a})\leq u(e_b).\label{th4-eq-bound-eb-2}
	\end{eqnarray*}
\textbf{About the path-length constraints,} we have
 \begin{eqnarray*}\tilde w(P_0)&=&\hat w(P_{v_1,CA(t_0)})+w(P_{CA(t_0),v_a})+\Delta \hat u(P_{CA(t_0),v_a})\\&&+w(e_b)+
		D^0-\Delta \hat u(P_{CA(t_0),v_a})+w(P_{v_b,t_0})\\&=&\hat w(P_0)+\Delta \hat u(P_{CA(t_0),v_a})+D^0-\Delta \hat u(P_{CA(t_0),v_a})=D.
	\end{eqnarray*} For other $t_k\in \ddot Y\backslash\{t_0\}$, we have $\tilde w(P_k)=\hat w(P_k)\geq D $. Hence, the feasibility holds.
    
 We next verify the cost and the optimality.  For $e\in P_{CA(t_0),v_a}$,
	\begin{eqnarray*}
	c(e)\big|\tilde w(e)-w(e)\big|=c(e)\Delta \hat u(e)=c(e)\min\Big\{\frac{C(\hat w)}{c(e)},\,\Delta u(e)\Big\}\leq C(\hat w).
	\end{eqnarray*}
	For $e_b$, since $0\leq \tilde w(e_b)-w(e_b)=D^0-\Delta \hat u(P_{CA(t_0),v_a})\leq \Delta \hat u(e_b)$ by (\ref{eq-th-case1_2}), we have
	\begin{eqnarray*}
	c(e_b)\big|\tilde w(e_b)-w(e_b)\big|\leq c(e_b)\Delta \hat u(e_b)\leq C(\hat w).
	\end{eqnarray*}
	For other edges, we have  $c(e)\big|\tilde w(e)-w(e)\big|=c(e)\big|\hat w(e)-w(e)\big|\leq C(\hat w)$. Therefore, $C(\tilde w)\leq C(\hat w)$. By Lemma \ref{lem-lowbd-hard}, we have 
	$C(\tilde w)\geq C(\hat w)$. Then $C(\tilde w)= C(\hat w)$ and $\tilde w$ is an optimal solution of the problem \textbf{RIOVSPT$\bm{_{\infty}(\ddot T)}$}.\hb

	\end{proof}

\subsection{The case when $\Delta \hat u(\chi_{t_0})<D^0$.}\label{subsec-case-up}
In this case $D^0=D-\hat w(P_0)>\Delta\hat u(\chi_{t_0})\geq 0$,that is, $\hat w(P_0)<D$. It means  even if every edge of the chain $\chi_{t_0}$ receives its free upgrade $\Delta\hat u(e)$, the length of $P_0$ is still smaller than $D$. Thus the length of $P_0$ cannot reach $D$ at the current cost level. 
 We next determine directly how much length can be obtained on $P_0$ under a per-edge cost budget $\lambda$. For any $\lambda\geq 0$, define the \textbf{achievable length function} of the path $P_0$ by
\begin{equation}\label{eq-Ulam}
L(\lambda):=\sum_{e\in P_0}\min\Big\{u(e),\ w(e)+\frac{\lambda}{c(e)}\Big\}.
\end{equation}
Clearly, $L(\lambda)$ is the maximal length of $P_0$ attainable by a weight vector whose per-edge adjustment cost does not exceed $\lambda$. Its relevant properties are summarized below, the proof being immediate from the definition.

\begin{lem}\label{lem-Uprop}
The function $L$ is continuous, nondecreasing and piecewise linear on
$[0,+\infty)$. It is strictly increasing on $[0,\bar\mu]$ and stays
constantly equal to $u(P_0)$ on $[\bar\mu,+\infty)$, where
$\bar\mu:=\max_{e\in P_0}c(e)\Delta u(e)$. Moreover, $L(0)=w(P_0)$.
\end{lem}

Since the problem (\textbf{RIOVSPT$\bm{_\infty}$}) is feasible, Theorem \ref{thm-fea=RIOVSPT_infty} gives $D\leq u(P_0)$, while in the present case $D>\hat w(P_0)\geq w(P_0)=L(0)$. Hence the value
\begin{equation}\label{eq-lamstar}
\lambda^*:=\min\{\lambda\geq 0\mid L(\lambda)\geq D\}
\end{equation}
is well defined, and the continuity of $L$ implies $L(\lambda^*)=D$.
The next lemma shows that in this case the new cost level necessarily exceeds $C(\hat w)$.

\begin{lem}\label{lem-lamstar-gt}
If $\Delta\hat u(\chi_{t_0})<D^0$, then $\lambda^*>C(\hat w)$.
\end{lem}
\begin{proof}
For $e\in P_0\setminus\chi_{t_0}$, we have $\min\left\{u(e),w(e)+\frac{C(\hat w)}{c(e)}\right\}=\hat w(e).$ \\
For $e\in\chi_{t_0}$ we have $\hat w(e)=w(e)$, and hence
$$\min\left\{u(e),w(e)+\frac{C(\hat w)}{c(e)}\right\}=w(e)+\min\left\{\frac{C(\hat w)}{c(e)},\Delta u(e)\right\}=\hat w(e)+\Delta\hat u(e).$$
 Summing over $e\in P_0$ yields
\begin{eqnarray*}
L\big(C(\hat w)\big)&=&\hat w(P_0\backslash \chi_{t_0})+\sum_{e\in \chi_{t_0}}(\hat w(e)+\Delta \hat u(e))\\&=&\hat w(P_0)+\Delta\hat u(\chi_{t_0})<\hat w(P_0)+D^0=D=L(\lambda^*).
\end{eqnarray*}
 If $\lambda^*\leq C(\hat w)$, the monotonicity of $L$ would imply $L(\lambda^*)\leq L(C(\hat w))<D=L(\lambda^*)$, a contradiction. Hence $\lambda^*>C(\hat w)$. \hb
\end{proof}

We next focus on computing the value $\lambda^*$. By Lemma~\ref{lem-lamstar-gt}, the search may start from $C(\hat w)$. Sort the distinct values of
$\{c(e)\Delta u(e)\mid e\in P_{0}, c(e)\Delta u(e)>C(\hat w)\}$
in ascending order:
\begin{eqnarray}
\lambda_{1}:=c(e_{\alpha_{1}})\Delta u(e_{\alpha_{1}})<\cdots
<\lambda_{m}:=c(e_{\alpha_{m}})\Delta u(e_{\alpha_{m}}),\label{eq:Sequence-up}
\end{eqnarray}
and set $\lambda_{0}:=C(\hat w)$. Since $L(\lambda_0)=\hat w(P_0)+\Delta\hat u(\chi_{t_0})<D$   and $L(\lambda_{m})=u(P_{0})\ge D$, a binary
search over $\lambda_{0},\dots,\lambda_{m}$ finds a unique index $i$
with $L(\lambda_{i})<D\le L(\lambda_{i+1})$.

Each edge $e\in P_{0}$ takes the value $u(e)$ if $c(e)\varDelta u(e)\le\lambda_{i}$, and $w(e)+\lambda_{i}/c(e)$ otherwise.  Hence $L(\lambda_{i})$ can be directly computed as follows.
\begin{equation}\label{eq-Laffine}
L(\lambda_j)=\sum_{e\in E_j}u(e)+\sum_{e\in P_0\setminus E_j}\Big(w(e)+\frac{\lambda_j}{c(e)}\Big),~\text{where } E_j:=\{e\in P_0\mid c(e)\Delta u(e)\leq\lambda_j\}.
\end{equation}
Then $\lambda^*\in(\lambda_i,\lambda_{i+1}]$. We have  $E_<^*:=\{e\in P_0\mid c(e)\Delta u(e)<\lambda^*\}=\{e\in P_0\mid c(e)\Delta u(e)\leq\lambda_i\}$ and $E_\ge^*:=\{e\in P_0\mid c(e)\Delta u(e)\ge\lambda^*\}=\{e\in P_0\mid c(e)\Delta u(e)>\lambda_i\}$. 
 Solving $L(\lambda^*)=\sum_{e\in E_<^*}u(e)+\sum_{e\in E_\ge^*}\Big(w(e)+\frac{\lambda^*}{c(e)}\Big)=D$  and we can obtain 
\begin{equation}\label{eq-lamstar-closed}
\lambda^*=\frac{\displaystyle D-w(P_0)-\sum_{e\in E_<^*}\Delta u(e)}{\displaystyle\sum_{e\in E_\ge^*}\frac{1}{c(e)}}.
\end{equation}

No feasible solution costs less than $\lambda^{*}$. Indeed, the bounds
and the budget constraint of any feasible $\tilde w$ give, for every
$e\in P_{0}$,
\[
\tilde w(e)\le\min\Big\{u(e),\,w(e)+\frac{C(\tilde w)}{c(e)}\Big\},
\]
so $D=\tilde w(P_{0})\le L(C(\tilde w))$, and hence
$C(\tilde w)\ge\lambda^{*}$ by the minimality in (\ref{eq-lamstar}).

We can now obtain an optimal solution of
\textbf{(RIOVSPT$\bm{_\infty}(\ddot T)$)} as follows.

\begin{thm}\label{thm-opt-up}
If $\Delta\hat u(\chi_{t_0})<D^0$,  then
\begin{equation}\label{eq-wstar-up}
\tilde w(e):=\left\{\begin{array}{ll}
w(e)+\min\Big\{\dfrac{\lambda^*}{c(e)},\,\Delta u(e)\Big\}, & e\in P_0\\
\hat w(e), & e\in\ddot T\setminus P_0
\end{array}\right.
\end{equation}
is an optimal solution of the subproblem \emph{\textbf{RIOVSPT$\bm{_\infty}$}$\bm{(\ddot T)}$}, where $\lambda^*$ is computed as (\ref{eq-lamstar-closed}).  The optimal value is $C(\tilde w)=\lambda^*$.
\end{thm}
\begin{proof}
\textbf{Feasibility.} 
For $e\in P_0$ we have $$l(e)\leq w(e)\leq\tilde w(e)=\min\left\{u(e),w(e)+\frac{\lambda^*}{c(e)}\right\}\leq u(e).$$ For  $e\in\ddot T\setminus P_0$ we have $l(e)\leq w(e)\leq\hat w(e)=\tilde w(e)\leq u(e)$. Moreover,
\begin{equation*}
\tilde w(P_0)=\sum_{e\in P_0}\min\Big\{u(e),\,w(e)+\frac{\lambda^*}{c(e)}\Big\}=L(\lambda^*)=D.
\end{equation*}
For $t_k\in\ddot Y\setminus\{t_0\}$ and $e\in P_0\cap P_k$, we
have, by Lemma~\ref{lem-lamstar-gt},
\begin{equation*}
\tilde w(e)=\min\Big\{u(e),\,w(e)+\frac{\lambda^*}{c(e)}\Big\}\geq\min\Big\{u(e),\,w(e)+\frac{C(\hat w)}{c(e)}\Big\}=\hat w(e).
\end{equation*}
While for $e\in P_k\setminus P_0$, we have $\tilde w(e)=\hat w(e)$. Hence $\tilde w(P_k)\geq\hat w(P_k)\geq D$  and $\tilde w$ is a feasible solution of the problem \textbf{RIOVSPT$\bm{_\infty}$}$\bm{(\ddot T)}$.

\textbf{Optimality.} On the one hand, 
for $e\in P_0$,
we have $c(e)\big(\tilde w(e)-w(e)\big)=\min\{\lambda^*,\,c(e)\Delta u(e)\}\leq\lambda^*$. 
For $e\in\ddot T\setminus P_0$, we have
$c(e)\big(\tilde w(e)-w(e)\big)=c(e)\big(\hat w(e)-w(e)\big)\leq C(\hat w)<\lambda^*$
by Lemma \ref{lem-lamstar-gt}. Therefore $C(\tilde w)\leq\lambda^*$.

On the other hand, no feasible solution costs less than $\lambda^{*}$. Indeed, the bounds
and the budget constraint of any feasible $\tilde w$ imply that , for every
$e\in P_{0}$,
\[
\tilde w(e)\le\min\Bigl\{u(e),\,w(e)+\frac{C(\tilde w)}{c(e)}\Bigr\}.
\]
Summing over $P_{0}$ yields $D=\tilde w(P_{0})\le L(C(\tilde w))$, and
hence $C(\tilde w)\ge\lambda^{*}$ by (\ref{eq-lamstar}). Together with
$C(\tilde w)\le\lambda^{*}$ proved above, 
$C(\tilde w)=\lambda^{*}$ holds.\hb
\end{proof}

\subsection{The case when $\Delta \hat l(\chi_{t_0})> D^0$. }
The remaining case is $\Delta\hat l(\chi_{t_0})>D^0$.  At the fixed level $C(\hat w)$, shortening $\chi_{t_0}$ alone
cannot bring the length of $P_0$ down to $D$. A larger budget is
needed, or edges outside $\chi_{t_0}$ on $P_0$ must also be adjusted,
or both. This is the most significant and difficult case of the problem
\textbf{(RIOVSPT$\bm{_\infty}(\ddot T)$)}. For convenience,
we call the paths $P_k$ with $t_k\in\ddot Y\setminus\{t_0\}$ the side
paths.

The analysis proceeds in four steps. Section~\ref{sss-spit}
introduces a parametric subproblem and solves it exactly by an
iterative greedy procedure. Section~\ref{sss-V} characterizes
its value function $V(M)$ and reduces the problem to the threshold
$M^{*}$. Section~\ref{sss-margins} tests the attainability of a
budget through truncated-bound margins. Section~\ref{sss-final}
locates the optimal budget in closed form and returns an optimal
solution.
\subsubsection {The parametric subproblem and its greedy solver.}\label{sss-spit}
Since $\Delta \hat l(\chi_{t_0})> D^0$, the length of $P_0$ must be reduced by further
adjustments.  We
determine their minimal extra cost beyond $C(\hat w)$. To this end we
introduce a parametric subproblem. It minimizes the length of $P_{0}$
under an extra budget. An iterative greedy procedure solves it exactly.

We now prepare the construction. If  $\Delta \hat l(\chi_{t_0})> D^0$, we first adjust $\hat w$ on $\chi_{t_0}$ by its free downgrading at the level $C(\hat w)$:
\begin{equation}\label{eq-free-down}
\hat w(e):=w(e)-\min\Big\{\frac{C(\hat w)}{c(e)},\,-\Delta l(e)\Big\}\;=\;w(e)+\Delta\hat l(e),\qquad \text{for~}e\in\chi_{t_0}
\end{equation}
and keep $\hat w(e)$ unchanged elsewhere. Note that $l(e)\le\hat w(e)\le w(e)$ for $e\in\chi_{t_0}$. 
After this adjustment the length of $P_0$ still exceeds $D$. Indeed,
$\hat w(P_0)-D=\Delta\hat l(\chi_{t_0})-D^0>0$ by the case condition.
This excess must be removed by further adjustments. By
Lemma~\ref{lem-lowbd-hard}, their cost has the form $C(\hat w)+M$ with
$M\ge0$. We introduce a parametric subproblem at the level
$C(\hat w)+M$ and study how its optimum varies with $M$.

In this subsection, the reference vector is the pre-adjusted $\hat w$ specified above. For
any extra cost $M\ge0$, the subproblem $\mathbf{SPIT}_{\infty}^{0}(M)$
seeks a weight vector $\tilde w^{M}$ to minimize the length of $P_0$ on the premise that 
the total adjustment cost under $l_\infty$ norm does not exceed the total budget $C(\hat w)+M$ and 
every path $P_j$ with $t_j\in\ddot Y\setminus\{t_0\}$ still has length at
least $D$. The mathematical model is as follows.
\begin{eqnarray}\label{SPITM}
&\min_{\tilde w^M}& \tilde w^M(P_0)\nonumber\\
\textbf{SPIT$_\infty^0$(M)}&s.t.& \max_{e\in \ddot T}c(e)|\tilde w^M(e)-w(e)|\leq C(\hat w)+M, \\
&& \tilde w^M(P_j)\geq D,\quad t_j\in \ddot Y\backslash \{t_0\},\nonumber\\
&&l(e)\leq \tilde w^M(e)\leq u(e),\quad e\in \ddot T.\nonumber
\end{eqnarray}

The subproblem is solved by an iterative greedy procedure. Its initial
state and update rules are specified next and then summarized as an
algorithm, together with the optimality analysis.

The construction starts from the following initial solution, which upgrades
the edges of $\ddot T\setminus P_0$ as much as the level $C(\hat w)+M$ allows
and downgrades the edges of $\chi_{t_0}$ accordingly:
\begin{equation}\label{tildew1}
\tilde w^1(e):=\left\{\begin{array}{lr}
\hat w(e)+\min\left\{\dfrac{M}{c(e)},\,u(e)-\hat w(e)\right\},&e\in\ddot T\backslash P_0\\[2mm]
\hat w(e),&e\in P_0\backslash\chi_{t_0}\\[2mm]
\hat w(e)-\min\left\{\dfrac{M}{c(e)},\,\hat w(e)-l(e)\right\},&e\in\chi_{t_0}
\end{array}\right..
\end{equation}

\begin{lem}\label{lem-w1fea}
For any $M\geq 0$, the solution $\tilde w^1$ as (\ref{tildew1})
is a feasible solution of the problem \textbf{SPIT$_\infty^0$(M)}.
\end{lem}
\begin{proof}
\textbf{About the bounds.} Recall that $w(e)\leq\hat w(e)\leq u(e)$ for $e\in\ddot T\backslash\chi_{t_0}$ and $l(e)\leq\hat w(e)\leq w(e)$ for $e\in\chi_{t_0}$ after the pre-adjustment (\ref{eq-free-down}).
 For $e\in\ddot T\backslash P_0$,we have 
$
u(e)\geq\tilde w^1(e)=\hat w(e)+\min\Big\{\frac{M}{c(e)},\,u(e)-\hat w(e)\Big\}\geq\hat w(e)\geq w(e)\geq l(e).
$\\
For $e\in P_0\backslash\chi_{t_0}$, $l(e)\leq\hat w(e)=\tilde w^1(e)\leq u(e)$. For $e\in\chi_{t_0}$,\\
$
l(e)=\hat w(e)-\big(\hat w(e)-l(e)\big)\leq\tilde w^1(e)=\hat w(e)-\min\Big\{\frac{M}{c(e)},\,\hat w(e)-l(e)\Big\}\leq \hat w(e)\leq u(e).
$
\textbf{About the cost.} For $e\in\ddot T\backslash P_0$,\\
$
c(e)\big|\tilde w^1(e)-w(e)\big|=c(e)\big(\hat w(e)-w(e)\big)+c(e)\min\Big\{\frac{M}{c(e)},\,u(e)-\hat w(e)\Big\}\leq C(\hat w)+M.
$
For $e\in P_0\backslash\chi_{t_0}$, $c(e)\big|\tilde w^1(e)-w(e)\big|=c(e)\big(\hat w(e)-w(e)\big)\leq C(\hat w)\leq C(\hat w)+M$.\\
For $e\in\chi_{t_0}$,\\
$
c(e)\bigl|\tilde w^{1}(e)-w(e)\bigr|
=c(e)\bigl(w(e)-\hat w(e)\bigr)
 +c(e)\min\Bigl\{\frac{M}{c(e)},\,\hat w(e)-l(e)\Bigr\}
\le C(\hat w)+M.
$\\
 \textbf{About the path constraints.} For  $t_k\in\ddot Y\backslash\{t_0\}$,\\
$
\tilde w^1(P_k)=\hat w(P_k\cap P_0)+\sum_{e\in P_k\backslash P_0}\Big(\hat w(e)+\min\Big\{\frac{M}{c(e)},\,u(e)-\hat w(e)\Big\}\Big)\geq\hat w(P_k)\geq D.
$
Hence $\tilde w^1$ is a feasible solution of the problem \textbf{SPIT$_\infty^0$(M)}. \hb
\end{proof}
Starting from $\tilde w^1$, the procedure shortens $P_0$ iteratively.
In each iteration it processes an initial subpath of the unprocessed
part of $P_0$. The edges of this subpath are decreased to their lowest
values attainable at the cost $C(\hat w)+M$. For convenience, for $e\in P_0$,
denote by
\begin{equation}\label{eq-target}
\underline w(e):=w(e)-\min\Big\{-\Delta l(e),\,\frac{C(\hat w)+M}{c(e)}\Big\} 
\end{equation}
the lowest value of edge $e\in P_0$ attainable at the cost $C(\hat w)+M$. Then for any
subpath $P\subseteq P_0$,  let
\begin{equation}\label{eq-room}
R(P):=\sum_{e\in P}\big(\hat w(e)-\underline w(e)\big)=\sum_{e\in P}\Big(\hat w(e)-w(e)+\min\Big\{-\Delta l(e),\frac{C(\hat w)+M}{c(e)}\Big\}\Big)
\end{equation}
be  the remaining shortening potential of $P$ from $\hat w$ at this
level. Note that $l(e)\leq\underline w(e)\leq w(e)$,
$c(e)\big(w(e)-\underline w(e)\big)\leq C(\hat w)+M$ and
$\underline w(e)\leq\hat w(e)$ for every $e\in P_0$, so that $R(P)\geq 0$.
Moreover, notice that $\tilde w^1(e)=\underline w(e)$ for $e\in\chi_{t_0}$.

For $v\in P_0$, denote by $d_0(v):=|E(P_{v_1,v})|$ the depth of
$v$ on $P_0$, which is the number of edges from the root $v_1$ to $v$. For each leaf $t_k\in\ddot Y\setminus\{t_0\}$, let $v^{*}(t_k)$ be the
\textbf{merge node} of $t_k$, which is the deepest intersection node on
$P_0\cap P_k$. 

We now describe one iteration of the procedure. It processes $P_0$ from
the root towards $t_0$. At the beginning of the $\tau$-th iteration, the
procedure maintains four objects as follows. \textbf{First,} the processed node $v^{\tau-1}\in P_0$ is the deepest node of $P_0$
whose values have been settled. Every edge of
$P_{v_1,v^{\tau-1}}$ has received its final value. 
Only the part of $P_0$ deeper than $v^{\tau-1}$ can still be modified.
A leaf $t_k\in\ddot Y\setminus\{t_0\}$ is still affected by the
forthcoming adjustments exactly when its root path shares an edge with
this part. We call such a leaf \textbf{active} after processing 
$v^{\tau-1}$. Equivalently, its merge node is strictly deeper than
$v^{\tau-1}$, that is,
$d_0\bigl(v^{*}(t_k)\bigr)>d_0(v^{\tau-1})$.
\textbf{Second,} the active set
$Y^\tau\subseteq\ddot Y\setminus\{t_0\}$ consists of the active
leaves. Hence, later updates can still affect the lengths of their
root-leaf paths. \textbf{Third,} the current vector
$\tilde w^\tau$ holds the weights fixed so far.\textbf{ Fourth,} the witness list $\mathcal T$ is an ordered log of the iterations.
Each entry corresponds to a leaf whose root-leaf path length has just been
reduced to exactly $D$. It records the leaf, its merge node  and the
edge at which this value is reached.

Initialization: $\tau:=1$,  $v^0:=v_1$, $Y^1:=\ddot Y\setminus\{t_0\}$, $\tilde w^1$ is given
by (\ref{tildew1})  and $\mathcal T:=\emptyset$. One invariant is maintained throughout. At the beginning of the $\tau$-th iteration, the weight of the edge $e\in P_{v^{\tau-1},CA(t_0)}$ on the unprocessed part of $P_0$ still equals $\hat w(e)$:
\begin{equation}\label{eq-invariant}
\tilde w^\tau(e)=\hat w(e), e\in P_{v^{\tau-1},CA(t_0)}.
\end{equation}
It holds initially by (\ref{tildew1}) and carries over to each
following iteration as every iteration modifies only its own
subpath $P_{v^{\tau-1},v^\tau}$.

While $Y^\tau\neq\emptyset$, the $\tau$-th iteration proceeds in the following three steps.

\textbf{Step 1:} select the active leaf $t^\tau\in Y^\tau$ with the currently shortest root-leaf path 
and the subpath $P_{v^{\tau-1},v^\tau}$ to be processed in this iteration, where $v^\tau:=v^*(t^\tau)$ is the merge node of $t^\tau$. 
 By (\ref{eq-invariant}),
every edge of this subpath still equals $\hat w(e)$ at the beginning
of the iteration. So the amount by which it can still be shortened at
the present level is exactly $R(P_{v^{\tau-1},v^{\tau}})$ defined in
(\ref{eq-room}).

\textbf{Step 2:} decrease the weights of the edges on the subpath $P_{v^{\tau-1},v^{\tau}}$. Two cases occur, depending on whether the
excess $D^{\tau}:=\tilde w^{\tau}(P_{t^{\tau}})-D$ can be exhausted
within this subpath.
If
$R\bigl(P_{v^{\tau-1},v^{\tau}}\bigr)<D^{\tau}$, the excess cannot be
exhausted within this subpath. Every edge of it is set to its  lowest value 
$\underline w(e)$. The length of every active leaf's root-leaf path
remains strictly above $D$. The new vector is   
\begin{equation}\label{tildewtauand1-1}
\tilde w^{\tau+1}(e):=\left\{\begin{array}{lr}\underline w(e),&e\in P_{v^{\tau-1},v^\tau}\\ \tilde w^\tau(e),&\text{otherwise}\end{array}\right..
\end{equation}
   Otherwise the excess is exhausted within this
subpath. Let $e_b=(v_a,v_b)\in P_{v^{\tau-1},v^{\tau}}$ be the first edge of the subpath  satisfying
\begin{equation}\label{eq-cross}
R\bigl(P_{v^{\tau-1},v_a}\bigr)\le D^\tau\le
R\bigl(P_{v^{\tau-1},v_b}\bigr).
\end{equation}
All edges of $P_{v^{\tau-1},v_a}$ are set to their floors. The edge $e_b$ 
is set to $\hat w(e_b)+\delta$ with
\begin{equation}\label{eq-delta}
\delta:=-D^\tau+R\bigl(P_{v^{\tau-1},v_a}\bigr).
\end{equation}
The edges of $P_{v_b,v^\tau}$ are left unchanged. The new vector is
\begin{equation}\label{tildewtauand1-2}
\tilde w^{\tau+1}(e):=\left\{\begin{array}{lr}\underline w(e),&e\in P_{v^{\tau-1},v_a}\\ \hat w(e_b)+\delta,&e=e_b\\ \tilde w^\tau(e),&\text{otherwise}\end{array}\right..
\end{equation}
The total shortening of this step is then exactly $D^\tau$. Hence the
root-leaf path length of $t^\tau$ is decreased to exactly $D$. The edge $e_b$ is called the \textbf{partial edge} of $t^\tau$. The triple
$(t^\tau,v^\tau,e_b)$ is appended to $\mathcal T$. 

\textbf{Step 3: update the state.} Update
$Y^{\tau+1}:=\{t\in Y^\tau:d_0(v^*(t))>d_0(v^\tau)\}$ and  $\tau:=\tau+1$. The invariant (\ref{eq-invariant}) carries over to the
next iteration.

The following lemma is the inductive step: one iteration maintains
feasibility and shortens $P_0$ as required.

\begin{lem}\label{lem-iter}
Suppose that $\tilde w^\tau$ is feasible for
$\mathbf{SPIT}_{\infty}^{0}(M)$ at the beginning of the $\tau$-th
iteration. 

\emph{(1)} If $R\bigl(P_{v^{\tau-1},v^\tau}\bigr)<D^\tau$, then
$\tilde w^{\tau+1}$ defined as (\ref{tildewtauand1-1}) is feasible for
$\mathbf{SPIT}_{\infty}^{0}(M)$ with
$\tilde w^{\tau+1}(P_0)=\tilde w^\tau(P_0)-R\bigl(P_{v^{\tau-1},v^\tau}\bigr)\le\tilde w^\tau(P_0)$
and $\tilde w^{\tau+1}(P_{t^\tau})>D$.

\emph{(2)} Otherwise, let $e_b=(v_a,v_b)\in
P_{v^{\tau-1},v^\tau}$ be the edge satisfying (\ref{eq-cross}) and
$\tilde w^{\tau+1}$ defined as (\ref{tildewtauand1-2}) is feasible for
$\mathbf{SPIT}_{\infty}^{0}(M)$  with
$\tilde w^{\tau+1}(P_0)=\tilde w^\tau(P_0)-D^\tau\le\tilde w^\tau(P_0)$
and $\tilde w^{\tau+1}(P_{t^\tau})=D$.
\end{lem}

\begin{proof}
By the invariant (\ref{eq-invariant}), $\tilde w^\tau(e)=\hat w(e)$ for every $e\in P_{v^{\tau-1},v^\tau}$.     Hence,  each update decreased it by at most $\hat w(e)-\underline w(e)$.
 Also, $D^\tau=\min_{t\in Y^\tau}\bigl(\tilde w^\tau(P_t)-D\bigr)$, so
every active leaf has length at least $D+D^\tau$  with equality at
$t=t^\tau$. 

(1) Feasibility. For $e\in P_{v^{\tau-1},v^\tau}$, we have $l(e)\le\tilde w^{\tau+1}(e)=\underline w(e)\le w(e)\le u(e)$ and 
$c(e)\bigl|\tilde w^{\tau+1}(e)-w(e)\bigr|
=c(e)\bigl(w(e)-\underline w(e)\bigr)\le C(\hat w)+M.$
Every other edge keeps its bounds and cost, since
$\tilde w^{\tau+1}(e)=\tilde w^\tau(e)$ and $\tilde w^\tau$ is
feasible.

For 
$t_j\in\ddot Y\setminus\{t_0\}$, if $t_j$ is no longer active at the
beginning of this iteration, its merge node lies no deeper than
$v^{\tau-1}$. Hence $P_j\cap P_{v^{\tau-1},v^\tau}=\emptyset$ and
$\tilde w^{\tau+1}(P_j)=\tilde w^\tau(P_j)\ge D$. If $t_j$ is still
active, its path length decreases by at most the total shortening of this
iteration and thus
\begin{equation*}
\tilde w^{\tau+1}(P_j)\ge\tilde w^\tau(P_j)-R\bigl(P_{v^{\tau-1},v^\tau}\bigr)
>\tilde w^\tau(P_j)-D^\tau\ge D.
\end{equation*}
Then the feasibility holds. 

In this case every edge of the subpath changes by exactly
$\hat w(e)-\underline w(e)$. The total shortening is therefore
$R\bigl(P_{v^{\tau-1},v^\tau}\bigr)$. Since $P_{t^\tau}$ contains the
whole subpath, we have 
$
\tilde w^{\tau+1}(P_{t^\tau})=\tilde w^\tau(P_{t^\tau})-R\bigl(P_{v^{\tau-1},v^\tau}\bigr)
=D+D^\tau-R\bigl(P_{v^{\tau-1},v^\tau}\bigr)>D$. As only the edges of this subpath are modified on $P_0$, so
$
\tilde w^{\tau+1}(P_0)=\tilde w^\tau(P_0)-R\bigl(P_{v^{\tau-1},v^\tau}\bigr)
\le\tilde w^\tau(P_0).$
Then  case (1) of the this lemma holds.

(2) Feasibility. For $e\in P_{v^{\tau-1},v_a}$ the
bounds and the cost follow as in case (1). Consider $e=e_b$. The two
inequalities of (\ref{eq-cross}) give
$-R(\{e_b\})\le\delta:=-D^\tau+R\bigl(P_{v^{\tau-1},v_a}\bigr)\le0$. Then
$$l(e_b)\le\underline w(e_b)=\hat w(e_b)-R(\{e_b\})
\le\tilde w^{\tau+1}(e_b)=\hat w(e_b)+\delta\le\hat w(e_b)\le u(e_b).$$
Additionally, since $\underline w(e_b)\le\tilde w^{\tau+1}(e_b)\le\hat w(e_b)$
and $\underline w(e_b)\le w(e_b)\le\hat w(e_b)$,
\begin{eqnarray*}
c(e_b)\bigl|\tilde w^{\tau+1}(e_b)-w(e_b)\bigr|
&\le&\max\Bigl\{c(e_b)\bigl(\hat w(e_b)-w(e_b)\bigr),\,
c(e_b)\bigl(w(e_b)-\underline w(e_b)\bigr)\Bigr\}\\
&\le& C(\hat w)+M.
\end{eqnarray*}
Every other edge keeps the value $\tilde w^\tau(e)$ and their bounds and cost constraints hold.

 The total shortening of
this iteration is $R\bigl(P_{v^{\tau-1},v_a}\bigr)+(-\delta)=D^\tau.$
Every path $P_j$( $t_j\in\ddot Y\setminus\{t_0\}$) shares a 
subpath with $P_0$, so its length decreases by at most the total shortening $D^\tau$
of this iteration. If $t_j$ is no longer active at the
beginning of this iteration, then  $P_j\cap P_{v^{\tau-1},v^\tau}=\emptyset$ and
$\tilde w^{\tau+1}(P_j)=\tilde w^\tau(P_j)\ge D$.  If $t_j$
is still active, then $\tilde w^{\tau+1}(P_j)\ge\tilde w^\tau(P_j)-D^\tau\ge D.$ Then the feasibility holds. 

Moreover, $P_{t^\tau}$ contains all the modified edges, so its length
decreases by exactly $D^\tau$ and then $
\tilde w^{\tau+1}(P_{t^\tau})=\tilde w^\tau(P_{t^\tau})-D^\tau=D.$
Since these are the only modified edges on $P_0$, then we have 
$
\tilde w^{\tau+1}(P_0)=\tilde w^\tau(P_0)-D^\tau\le\tilde w^\tau(P_0).
$
Therefore,  case (2) of the this lemma holds.\hb
\end{proof}

When the loop exits, $Y^\tau=\emptyset$, which means no active leaves. The processed part of $P_0$
ends at $v^{\tau-1}$. So the remaining path $P_{v^{\tau-1},CA(t_0)}$
shares no edge with any path $P_j$,
$t_j\in\ddot Y\setminus\{t_0\}$. Its edges can therefore be set to their lowest attainable
values without affecting any other paths. 
This is done in the terminal step
\begin{equation}\label{eq-leftover}
\tilde w^{\tau+1}(e):=\left\{\begin{array}{lr}\underline w(e),&e\in P_{v^{\tau-1},CA(t_0)}\\ \tilde w^\tau(e),&\text{otherwise}\end{array}\right..
\end{equation}
The whole procedure is summarized in Algorithm~\ref{alg-SPIT0}, where the labels $v^{*}(t)$ and $d_0$ can be precomputed in $O(n)$ time by one
depth-first traversal of $\ddot T$.

\begin{breakablealgorithm}
	\renewcommand{\thealgorithm}{1}
	\caption{$[\tilde w^M,\tilde w^M(P_0)]:=\textbf{SPIT}_\infty^0(M)(\ddot T, \ddot Y, w,\hat w, l, u, c, D,M,t_0)$}
	\label{alg-SPIT0}
	\begin{algorithmic}[1]
		\REQUIRE The gate subtree $\ddot T$, its leaf set $\ddot Y$, vectors $w,l,u,c$, the vector $\hat w$,  the given leaf $t_0$, the given values$D$ and $M$. 
		\ENSURE The optimal solution $\tilde w^M$  with  its objective value $\tilde w^M(P_0)$.
		\STATE Initialize $v^0:=v_1$, $\tau:=1$, $Y^1:=\ddot Y\backslash\{t_0\}$ and $\mathcal T:=\emptyset$.
		\STATE Construct $\tilde w^1$ by (\ref{tildew1}).
		\WHILE{$Y^\tau\neq\emptyset$}
		\STATE Let $D^\tau:=\min_{t\in Y^\tau}\left(\tilde w^\tau(P_t)-D\right)$. Let $t^\tau:=\arg\min_{t\in Y^\tau}\left(\tilde w^\tau(P_t)-D\right)$  and $v^{\tau}:=v^{*}(t^{\tau})$ be its merge node.
		\IF{$R\big(P_{v^{\tau-1},v^\tau}\big)<D^\tau$}
			\STATE Update $\tilde w^{\tau+1}$ by (\ref{tildewtauand1-1}).
		\ELSE
			\STATE Determine the edge $e_b=(v_a,v_b)\in P_{v^{\tau-1},v^\tau}$ satisfying (\ref{eq-cross}).  Compute $\delta$ by (\ref{eq-delta}) and update $\tilde w^{\tau+1}$ by (\ref{tildewtauand1-2}).  Append $(t^\tau,v^\tau,e_b)$ to $\mathcal T$.
		\ENDIF
		\STATE Update $Y^{\tau+1}:=\{t\in Y^\tau:d_0(v^*(t))>d_0(v^\tau)\}$ and $\tau:=\tau+1$. 
		\ENDWHILE
		\STATE  Update $\tilde w^{\tau+1}$ as (\ref{eq-leftover}).
		\RETURN ($\tilde w^M:=\tilde w^{\tau+1}$, $\tilde w^M(P_0)$)
	\end{algorithmic}
\end{breakablealgorithm}
The optimal solution $\tilde w^{M}$ of
$\mathbf{SPIT}_{\infty}^{0}(M)$ produced by
Algorithm~\ref{alg-SPIT0} is called the \textbf{canonical optimal
solution} of $\mathbf{SPIT}_{\infty}^{0}(M)$. We next prove that Algorithm \ref{alg-SPIT0} solves the subproblem exactly.
\begin{thm}\label{thm-SPITopt}
For any $M\ge0$, let $\tilde w^{M}$ be the output of
Algorithm~\emph{\ref{alg-SPIT0}}. Then $\tilde w^{M}$ is an optimal
solution of $\mathbf{SPIT}_{\infty}^{0}(M)$  and $\tilde w^{M}(P_{0})$
is the minimum length of $P_{0}$ among all feasible solutions.
\end{thm}
\begin{proof}
\textbf{Feasibility.} The initial vector $\tilde w^1$ is feasible by
Lemma~\ref{lem-w1fea}. The invariant (\ref{eq-invariant}) holds
throughout. Hence the conclusion of Lemma~\ref{lem-iter} holds at every iteration. In the terminal step (\ref{eq-leftover}), every modified edge $e$ is
set to $\underline w(e)$ with $l(e)\le\underline w(e)\le w(e)$ and then 
$c(e)\bigl(w(e)-\underline w(e)\bigr)\le C(\hat w)+M$. Moreover, no
path $P_j(t_j\in\ddot Y\setminus\{t_0\})$ contains any of these  edges,
otherwise its leaf would still be active and the loop would not
have terminated. Hence the output $\tilde w^M$ is feasible for
$\mathbf{SPIT}_{\infty}^{0}(M)$.

\textbf{Optimality.} Let $w'$ be any feasible solution of
$\mathbf{SPIT}_{\infty}^{0}(M)$. We show that $w'(P_{0})\ge\tilde
w^{M}(P_{0})$ in three steps. 

(1) We first show that for $e\in\ddot T\setminus P_{0}$,  $w'(e)\le\tilde w^{M}(e)$ holds.  The iterations only modify edges of
$P_{0}$. Hence $\tilde w^{M}(e)=\hat w(e)+\min\bigl\{\frac{M}{c(e)},\,
u(e)-\hat w(e)\bigr\}$ by (\ref{tildew1}). Suppose the minimum equals
$u(e)-\hat w(e)$. This happens in particular when $\hat w(e)=u(e)$.
Then $\tilde w^{M}(e)=u(e)\ge w'(e)$. Otherwise the definition of
$\hat w$ gives $\hat w(e)=w(e)+\frac{C(\hat w)}{c(e)}$. Hence
$\tilde w^{M}(e)=w(e)+\frac{C(\hat w)+M}{c(e)}$. If
$w'(e)>\tilde w^{M}(e)$, then
\begin{equation*}
c(e)\bigl|w'(e)-w(e)\bigr|\ge c(e)\bigl(w'(e)-w(e)\bigr)
>c(e)\bigl(\tilde w^{M}(e)-w(e)\bigr)=C(\hat w)+M,
\end{equation*}
contradicting the cost constraint of
$\mathbf{SPIT}_{\infty}^{0}(M)$.

(2)  Let $\mathcal T$ be the
witness list recorded by Algorithm~\ref{alg-SPIT0} during the
execution. We next show that for every witness $(t^{\tau},v^{*}(t^{\tau}),e_b)\in\mathcal T$, 
\begin{equation}\label{eq-sharedprefix}
w'\bigl(P_{v_1,v^{*}(t^{\tau})}\bigr)\ \ge\
\tilde w^{M}\bigl(P_{v_1,v^{*}(t^{\tau})}\bigr).
\end{equation}
Indeed, the iteration that recorded $t^{\tau}$ shortened its path by
exactly $D^{\tau}$, reducing its length to exactly $D$. Every later
iteration modifies only the edges of $P_0$ deeper than
$v^{*}(t^{\tau})$. Hence the final vector still satisfies
$\tilde w^{M}(P_{t^{\tau}})=D$. Since part (1) shows
$w'(P_{t^\tau}\setminus P_0)\le\tilde w^{M}(P_{t^\tau}\setminus P_0)$ 
and $w'(P_{t^\tau})\ge D$, then we have 
$$w'\bigl(P_{v_1,v^{*}(t^{\tau})}\bigr)
 =w'(P_{t^{\tau}})-w'(P_{t^{\tau}}\setminus P_0)
 \ge D-\tilde w^M(P_{t^{\tau}}\setminus P_0)
 =\tilde w^M\bigl(P_{v_1,v^{*}(t^{\tau})}\bigr).$$

(3) We finally show that $w'(P_0)\geq \tilde w^M(P_0)$. Note first that every feasible
$w'$ satisfies $w'(e)\ge\underline w(e)$ for $e\in P_0$. 
Let $\bar v:=v^{\tau^*-1}$, where $\tau^*$ is the value of $\tau$ at
which the loop exits. We distinguish two cases.

Suppose first that $\mathcal T\neq\emptyset$. Let $t_k$ be the last
witness leaf in $\mathcal T$. Then $P_0$ splits into three parts.
\textbf{The first part} $P_{v_1,v^{*}(t_k)}$ is covered by
(\ref{eq-sharedprefix}). Let $\tau_k$ be the iteration that recorded $t_k$. Take any edge $e$
of \textbf{the second part} $P_{v^{*}(t_k),\bar v}$. It was processed in some
iteration $\tau>\tau_k$. This iteration must fall under case (1) of
Lemma~\ref{lem-iter}: if it fell under case (2) of
Lemma~\ref{lem-iter}, it would have appended a witness later than
$t_k$, contradicting the choice of $t_k$. By the update rule
(\ref{tildewtauand1-1}) of that case, the edge $e$ was then set to
its lowest value $\underline w(e)$, and no later iteration touches it again.
 Hence, $\tilde w^M(e)=\underline w(e)\le w'(e)$. \textbf{The third part} $P_{\bar v,t_0}$ also satisfies
$\tilde w^M(e)=\underline w(e)\le w'(e)$ for every edge $e$. Its
edges are of two kinds. An edge on $\chi_{t_0}$ already equals
$\underline w(e)$ in the initial vector $\tilde w^1$ by
(\ref{tildew1}), and no iteration ever modifies it. Each remaining edge is set to $\underline w(e)$ by the terminal step
(\ref{eq-leftover}). Combining the three parts, we obtain
\begin{equation*}
\begin{aligned}
w'(P_0)
&=w'\bigl(P_{v_1,v^{*}(t_k)}\bigr)+w'\bigl(P_{v^{*}(t_k),\bar v}\bigr)
 +w'\bigl(P_{\bar v,t_0}\bigr)\\
&\ge\tilde w^M\bigl(P_{v_1,v^{*}(t_k)}\bigr)
 +\tilde w^M\bigl(P_{v^{*}(t_k),\bar v}\bigr)
 +\tilde w^M\bigl(P_{\bar v,t_0}\bigr)
=\tilde w^M(P_0).
\end{aligned}
\end{equation*}

Suppose that $\mathcal T=\emptyset$. Then every iteration fell under
case (1) of Lemma~\ref{lem-iter}, and every edge of $P_0$ lies either
on $\chi_{t_0}$ or in the subpath set to $\underline w$ by the
terminal step (\ref{eq-leftover}). Hence
$\tilde w^M(e)=\underline w(e)\le w'(e)$ for every $e\in P_0$.
Summing over $P_0$ gives $w'(P_0)\ge\tilde w^M(P_0)$.

In both cases $w'(P_0)\ge\tilde w^M(P_0)$. Hence $\tilde w^M$ is
optimal for $\mathbf{SPIT}_{\infty}^{0}(M)$.
\end{proof}

\subsubsection {The value function and the threshold $M^*$.}\label{sss-V}
For each individual $M$, the subproblem is now settled by Algorithm \ref{alg-SPIT0}. What is still missing is a global view over all budgets $M\ge 0$. We next develop it, which leads to a single threshold $M^{*}$.

We introduce a notation used throughout the rest of this paper. For
a real number $B\ge0$ and an edge $e\in\ddot T$, define the
\textbf{truncated bounds}:
\begin{equation}\label{eq-trunc}
l_{B}(e):=\max\Big\{l(e),\,w(e)-\frac{B}{c(e)}\Big\},\qquad
u_{B}(e):=\min\Big\{u(e),\,w(e)+\frac{B}{c(e)}\Big\}.
\end{equation}
They describe the feasible range of an edge value under a cost budget
$B$: the two requirements $c(e)|\tilde w(e)-w(e)|\le B$ and
$l(e)\le\tilde w(e)\le u(e)$ together are equivalent to the single
inequality $l_{B}(e)\le\tilde w(e)\le u_{B}(e)$. Lemma~\ref{lem-Vprop}
below will use (\ref{eq-trunc}) with $B=C(\hat w)+M$. The attainability
analysis afterwards treats $B$ as a free parameter. The following observation is immediate.
\begin{lem}\label{lem-budget}
For every $B\ge0$ and every edge $e\in\ddot T$,
\begin{equation}\label{eq-sandwich}
l(e)\ \le\ l_B(e)\ \le\ w(e)\ \le\ u_B(e)\ \le\ u(e).
\end{equation}
Moreover, a weight vector $\tilde w$ satisfies the bounds
$l(e)\le\tilde w(e)\le u(e)$ and the budget constraint
$c(e)\big|\tilde w(e)-w(e)\big|\le B$ for all $e\in\ddot T$ if and only if
it satisfies $l_B(e)\le\tilde w(e)\le u_B(e)$ for all $e\in\ddot T$.
\end{lem}

\begin{remark}\label{rem:kinks}
Each $l_{B}(e)$ is continuous and piecewise affine in $B$ with a single
kink at $B=c(e)\bigl(w(e)-l(e)\bigr)$, and each $u_{B}(e)$ is continuous
and piecewise affine in $B$ with a single kink at
$B=c(e)\bigl(u(e)-w(e)\bigr)$.
\end{remark}

With the subproblem \textbf{SPIT$_\infty^0$(M)} solved exactly, we can study how the minimum length of $P_0$ varies with the extra cost $M$, and thereby determine the optimal value of the problem \textbf{RIOVSPT$\bm{_{\infty}(\ddot T)}$}. For any $M\geq 0$, define the value function
\begin{equation}\label{eq-V}
V(M):=\tilde w^M(P_0),
\end{equation}
where $\tilde w^M$ is the canonical optimal solution of the problem \textbf{SPIT$_\infty^0$(M)}.  By Theorem \ref{thm-SPITopt}, $V(M)$ is the minimum length of $P_0$ over all feasible solutions of \textbf{SPIT$_\infty^0$(M)}.
\begin{lem}\label{lem-Vprop}
The value function $V$ in \emph{(\ref{eq-V})} has the following properties.

\emph{(1)} $V$ is nonincreasing on $[0,+\infty)$.

\emph{(2)} $V$ is continuous and piecewise affine on $[0,+\infty)$.

\emph{(3)} The problem \textbf{RIOVSPT$\bm{_{\infty}(\ddot T)}$} admits a feasible solution with cost at most $C(\hat w)+M$ if and only if $V(M)\leq D$.

\emph{(4)} $V(M)\leq D$ holds for every sufficiently large $M$.
\end{lem}
\begin{proof}
(1) Let $M\le M'$. The feasible set of
$\mathbf{SPIT}_{\infty}^{0}(M')$ contains that of
$\mathbf{SPIT}_{\infty}^{0}(M)$: the bound constraints and the side-path
constraints are the same in the two problems, while the cost budget
$C(\hat w)+M'$ is larger than $C(\hat w)+M$. Enlarging the feasible set can only lower the minimum, so $V(M')\le V(M)$.  Hence $V$ is nonincreasing
on $[0,+\infty)$.

(2) We determine $V(M)$ by a closed formula. Let $B=C(\hat w)+M$. We first show that no feasible vector can make $P_0$ shorter than an
explicit quantity $H(B)$  and then construct a feasible vector attaining
$H(B)$.
By (\ref{eq-trunc}) and Lemma \ref{lem-budget}, define 
\begin{equation}\label{eq-H}
H(B):=\max\left\{l_B(P_0),\ \max_{t_k\in\ddot Y\setminus\{t_0\}}
\bigl[D+l_B(P_0\setminus P_k)-u_B(P_k\setminus P_0)\bigr]\right\}.
\end{equation}

 Let $\tilde w$ be any feasible vector of
$\mathbf{SPIT}_{\infty}^{0}(M)$. 
We  first have $\tilde w(P_0)\ge l_B(P_0)$ by the edge bounds. 
 Then for every leaf $t_k\in \ddot Y\backslash\{t_0\}$, we have $\tilde w(P_k)\geq D$ and $l_B(e)\le\tilde w(e)\le u_B(e)$. Then  
\[
\tilde w(P_0)=\tilde w(P_k)-\tilde w(P_k\setminus P_0)+\tilde w(P_0\setminus
P_k)\ \ge\ D-u_B(P_k\setminus P_0)+l_B(P_0\setminus P_k),
\]
 Summarizing, every feasible vector $\tilde w$ of
$\mathbf{SPIT}_{\infty}^{0}(M)$ satisfies
\[
\tilde w(P_0)  \ge\ l_B(P_0)
 \text{~and~} 
\tilde w(P_0) \ge  D-u_B(P_k\setminus P_0)+l_B(P_0\setminus P_k)
\quad\bigl(t_k\in\ddot Y\setminus\{t_0\}\bigr).
\]
That is, $\tilde w(P_0)$ is no smaller than every term of (\ref{eq-H}), and
hence no smaller than their maximum. Then  for every feasible vector $\tilde w$, we have 
$\tilde w(P_0)\ \ge\ H(B)$. 
Therefore, by (\ref{eq-V}), we have 
\begin{equation}\label{eq-Hlower}
V(M)\ \ge\ H(B).
\end{equation}
What's more, $\hat w$ can also be seen as a feasible solution of  $\mathbf{SPIT}_{\infty}^{0}(M)$ with $C(\hat w)\leq B$, then  we have 
\begin{equation}\label{eq-Hchain}
H(B)  \le  \hat w(P_0) \le\ u_B(P_0).
\end{equation}

 We exhibit a feasible vector
$w'$ with $w'(P_0)=H(B)$, proving $V(M)\le H(B)$. Distribute the increment 
$\Delta:=H(B)-l_B(P_0)\leq u_B(P_0)-l_B(P_0)$ by (\ref{eq-Hchain}) to fill up the edges on the path $P_0$ from  $v_1$ to  $t_0$.
Then  there is a  edge $e_b=(v_a,v_b)\in
P_0$  satisfying
\begin{equation}\label{eq-crossB}
u_B(P_{v_1,v_a})-l_B(P_{v_1,v_a})\le\Delta\le
u_B(P_{v_1,v_b})-l_B(P_{v_1,v_b}).
\end{equation}
Let 
\begin{equation}\label{eq-construct}
w'(e):=\begin{cases}
u_B(e), & e\in\ddot T\setminus P_{v_a,t_0}\\
l_B(e_b)+\Delta-\bigl(u_B(P_{v_1,v_a})-l_B(P_{v_1,v_a})\bigr), & e=e_b\\
l_B(e), & e\in P_{v_b,t_0}
\end{cases}.
\end{equation}
Clearly,  $l_B(e)\le w'(e)\le u_B(e)$ holds for every edge.
We next show $w'(P_k)\ge D$ for every $t_k\in\ddot Y\setminus\{t_0\}$. Let
$v^*:=v^*(t_k)$ be the merge node of $t_k$, which is the last intersection node on $P_0 \cap P_k$. We have 
$w'(P_k\setminus P_0)=u_B(P_k\setminus P_0)$ and it remains to show
$w'(P_{v_1,v^*})\ge D-u_B(P_k\setminus P_0)$.

If $e_b\notin P_{v_1,v^*}$, then 
\[
w'(P_{v_1,v^*})=u_B(P_{v_1,v^*})\ \ge\ \hat w(P_{v_1,v^*})\ =\
\hat w(P_k)-\hat w(P_k\setminus P_0)\ \ge\ D-u_B(P_k\setminus P_0),
\]
If  $e_b\in P_{v_1,v^*}$, then $P_0\setminus P_k=P_{v^*,t_0}$ and we have $H(B)\geq D+l_B(P_{v^*,t_0})-u_B(P_k\backslash P_0)$ in (\ref{eq-H}). Thus
\[
\begin{aligned}
w'(P_{v_1,v^*})
&=u_B(P_{v_1,v_a})+l_B(e_b)+\Delta-\bigl(u_B(P_{v_1,v_a})-l_B(P_{v_1,v_a})\bigr)
  +l_B(P_{v_b,v^*})\\
&=l_B(P_{v_1,v^*})+\Delta
 =H(B)-l_B(P_{v^*,t_0})\ge D-u_B(P_k\setminus P_0).
\end{aligned}
\]

In both cases $w'(P_{v_1,v^*})\ge D-u_B(P_k\setminus P_0)$ and   we have 
$w'(P_k)\ge D$. Then  $w'$ is feasible for
$\mathbf{SPIT}_{\infty}^{0}(M)$. 
As
\begin{eqnarray*}
    w'(P_0)&=&u_B(P_{v_1,v_a})+l_B(e_b)+\Delta-\bigl(u_B(P_{v_1,v_a})-l_B(P_{v_1,v_a})\bigr)
+l_B(P_{v_b,t_0})\\&=&l_B(P_0)+\Delta=H(B),
\end{eqnarray*}

 we obtain
$V(M)\le w'(P_0)= H(B)$. Together with (\ref{eq-Hlower}), we have
\begin{eqnarray}
    V(M)=H(B).\label{eq-VH}
\end{eqnarray}

 We show that $V$ is continuous and piecewise affine on
$[0,+\infty)$. Consider first the building blocks of $H(B)$, namely
$l_B(e)$ and $u_B(e)$ defined in (\ref{eq-trunc}).
By Remark~\ref{rem:kinks}, each $l_{B}(e)$ and $u_{B}(e)$ is continuous
and piecewise affine, and so is every term of $H(B)$, being a finite sum
of such functions with coefficients $\pm1$. The kinks of each individual
term of $H(B)$ lie in the finite set (\ref{eq-kinks}).
\begin{equation}\label{eq-kinks}
\{0\}\cup\bigl\{c(e)\bigl(w(e)-l(e)\bigr),\,c(e)\bigl(u(e)-w(e)\bigr)\;\big|\;
e\in\ddot T\bigr\}.
\end{equation}
The pointwise maximum of a finite family of continuous piecewise affine functions is itself continuous and piecewise affine, hence so is $H(B)$. Finally, $B=C(\hat w)+M$ is affine in $M$, so $V$ inherits both properties
from $H$ through (\ref{eq-VH}): $V$ is continuous and piecewise affine on
$[0,+\infty)$.

(3) \textbf{Necessity.} Let $\tilde w$ be a feasible solution of
$\mathrm{RIOVSPT}_{\infty}(\ddot T)$ with $C(\tilde w)\le C(\hat w)+M$ and we
show that $V(M)\le D$. Obviously, the vector $\tilde w$ is also feasible for
$\mathbf{SPIT}_{\infty}^{0}(M)$. As $V(M)$ is
the minimum of the length of $P_0$ over all feasible solution for
$\mathbf{SPIT}_{\infty}^{0}(M)$ and thus we have 
$V(M)\ \le\ \tilde w(P_0)\ =\ D.$

\textbf{Sufficiency.} Assume $V(M)\le D$. It remains to produce a feasible solution of
$\mathrm{RIOVSPT}_{\infty}(\ddot T)$ with cost at most $C(\hat w)+M$. We
obtain it as a convex combination of two vectors, $\tilde w^{M}$ and $\hat w$, which are both feasible for
$\mathbf{SPIT}_{\infty}^{0}(M)$ with cost at most $C(\hat w)+M$. 
Recall that $ \tilde w^{M}$ is the canonical optimal solution of the problem $\mathbf{SPIT}_{\infty}^{0}(M)$ with 
$\tilde w^{M}(P_0)=V(M)\le D$ and $\hat w$ is the pre-adjusted 
vector based on the \text{canonical optimal solution} of the problem \textbf{MCSPIT$\bm{_\infty}$}$\bm{(\ddot T\backslash\chi_{t_0})}$ established in
(\ref{eq-free-down}) with $\hat w(P_0)>D$.  
The length of $P_0$  varies linearly from $\tilde w^{M}(P_0)\le D$ to
$\hat w(P_0)>D$, so some convex combination hits $D$ exactly. The
corresponding weight, obtained by solving
$\alpha\hat w(P_0)+(1-\alpha)\tilde w^{M}(P_0)=D$, is
\[
\alpha:=\frac{D-\tilde w^{M}(P_0)}{\hat w(P_0)-\tilde w^{M}(P_0)}
       =\frac{D-V(M)}{\hat w(P_0)-V(M)}.
\]
Notice that the numerator satisfies $D-\tilde w^{M}(P_0)  =  D-V(M) \ge  0$
and the denominator
satisfies $ \hat w(P_0)-\tilde w^{M}(P_0)  =  \hat w(P_0)-V(M)  > D-V(M)\ge 0.$ Then $\alpha\in[0,1)$ holds.
 Construct $\tilde w_{\alpha}:=\alpha\hat w+(1-\alpha)\tilde w^{M}$. We next show its feasibility. 
 About the bound, for every edge $e$, we have 
$l(e)\le\hat w(e)\le u(e)$  and 
$l(e)\le\tilde w^{M}(e)\le u(e)$, then 
$   l(e) =\alpha l(e)+(1-\alpha) l(e) \le  \alpha\hat w(e)+(1-\alpha)
\tilde w^{M}(e)  =  \tilde w_{\alpha}(e)  \le \alpha u(e)+(1-\alpha)
u(e)  =\ u(e). $\\
About the length of paths, We have $\tilde w_{\alpha}(P_0)=\alpha\hat w(P_0)+(1-\alpha)\tilde w^{M}(P_0)
      \ =\ D.$ \\
      For every $t_k\in\ddot Y\setminus\{t_0\}$,
     $\tilde w_{\alpha}(P_k)=\alpha\hat w(P_k)+(1-\alpha)\tilde w^{M}(P_k)
      \ \ge\ \alpha D+(1-\alpha)D\ =\ D. $
Then about its cost,    
         \begin{eqnarray*}
         c(e)\bigl|\tilde w_{\alpha}(e)-w(e)\bigr|
      & \le& \alpha c(e)\bigl|\hat w(e)-w(e)\bigr|
        +(1-\alpha) c(e)\bigl|\tilde w^{M}(e)-w(e)\bigr|\\
       & \le& \alpha C(\hat w)+(1-\alpha)\bigl(C(\hat w)+M\bigr)
        \le C(\hat w)+M \qquad(e\in\ddot T).
    \end{eqnarray*}
Hence,  $\tilde w_{\alpha}$ is the required feasible solution.

(4) Finally, we show that $V(M)\le D$ for every sufficiently large
$M$. Let $\tilde w$  be a feasible solution of the problem
$\mathrm{RIOVSPT}_{\infty}(\ddot T)$. Let
$M:=C(\tilde w)-C(\hat w)\geq 0$ by Lemma \ref{lem-lowbd-hard}. Then $\tilde w$ is also a feasible solution of cost at most $C(\hat w)+M$ of the problem  $\mathrm{RIOVSPT}_{\infty}(\ddot T)$. Hence from  (3), we have  $V(M)\le D$, and monotonicity (1) extends this
to every $M'\ge M$.\hb
\end{proof}

Lemma~\ref{lem-Vprop} characterizes the value function $V$. It is
nonincreasing, continuous and piecewise affine. Moreover, by
Lemma~\ref{lem-Vprop}(3), whether $V(M)\le D$ can be decided by testing
the budget $C(\hat w)+M$ for attainability. Hence determining the optimal
value of $\mathrm{RIOVSPT}_{\infty}(\ddot T)$ reduces to locating the
threshold where $V$ crosses the target $D$.

\begin{lem}\label{lem-optval}
The optimal value of the problem $\mathrm{RIOVSPT}_{\infty}(\ddot T)$
equals $C(\hat w)+M^{*}$, where
\begin{equation}\label{eq-Mstar}
M^{*}:=\min\bigl\{M\ge 0\mid V(M)\le D\bigr\} 
\end{equation}
and the minimum is attained.
\end{lem}
\begin{proof}
The minimum in (\ref{eq-Mstar}) is attained, because the set is
nonempty (Lemma~\ref{lem-Vprop}(4)), upward closed ($V$ is nonincreasing,
Lemma~\ref{lem-Vprop}(1)) and closed ($V$ is continuous,
Lemma~\ref{lem-Vprop}(2)), hence a closed ray $[M^{*},+\infty)$
containing its left endpoint $M^{*}$.

For the lower bound, Lemma \ref{lem-lowbd-hard} shows that no feasible solution has a cost
below $C(\hat w)$. Let   $\tilde w$ be a feasible solution of the problem $\mathrm{RIOVSPT}_{\infty}(\ddot T)$ and record its
cost as $C(\tilde w)=C(\hat w)+M'$ with $M'\ge0$. Then $\tilde w$ is
feasible for $\mathbf{SPIT}_{\infty}^{0}(M')$ and $\tilde w(P_{0})=D$. 
Hence $V(M')\le D$ by Lemma~\ref{lem-Vprop}(3) and then  $M'\ge M^{*}$ by (\ref{eq-Mstar}), that
is, $C(\tilde w)=C(\hat w)+M'\ge C(\hat w)+M^{*}$.

For the upper bound, $M^{*}$ itself satisfies $V(M^{*})\le D$, and
Lemma~\ref{lem-Vprop}(3) provides a feasible solution of cost at most
$C(\hat w)+M^{*}$. Combining both bounds, the optimal value equals
$C(\hat w)+M^{*}$.\hb
\end{proof}

 By continuity and
minimality, $V(0)>D$ implies $V(M^{*})=D$. In this case an optimal
solution is $\tilde w^{M^{*}}$, while otherwise the convex combination
constructed in the proof of Lemma~\ref{lem-Vprop}(3) is optimal. The
case $M^{*}=0$ is settled in the next lemma.

\begin{lem}\label{lem-trivial}
An optimal solution of $\mathrm{RIOVSPT}_{\infty}(\ddot T)$ is given as
follows. If $V(0)>D$, it is $\tilde w^{M^{*}}$. If $V(0)=\tilde w^0(P_0)\le D$, it is the
convex combination
\begin{equation}\label{eq-walpha}
\tilde w_{\alpha}:=\alpha\hat w+(1-\alpha)\tilde w^{0},
\qquad
\alpha:=\frac{D-\tilde w^{0}(P_{0})}{\hat w(P_{0})-\tilde w^{0}(P_{0})}.
\end{equation}
\end{lem}
\begin{proof}
We treat the two cases separately.
 
 If $V(0)>D$, we claim that $V(M^{*})=D$. First,
$V(M^{*})\le D$. This is because $M^{*}$ belongs to the set of
(\ref{eq-Mstar}). Second, $V(M^{*})$ cannot be strictly smaller than $D$.
Indeed, suppose $V(M^{*})<D$. By the continuity of $V$, there exists
$\varepsilon\in(0,M^{*})$ such that $V(M^{*}-\varepsilon)<D$. Then
$M^{*}-\varepsilon$ also belongs to the set of (\ref{eq-Mstar}). This
contradicts the minimality of $M^{*}$. Hence $V(M^{*})=D$.
Consequently, $\tilde w^{M^{*}}(P_{0})=D$. The other constraints of
$\mathrm{RIOVSPT}_{\infty}(\ddot T)$ hold as well and thus $\tilde w^{M^{*}}$
is feasible for $\mathrm{RIOVSPT}_{\infty}(\ddot T)$. Its cost is 
$C(\hat w)+M^{*}$. By Lemma~\ref{lem-optval}, this is the optimal value.
Therefore $\tilde w^{M^{*}}$ is optimal.

If $V(0)\le D$, then $0$ belongs to the set 
(\ref{eq-Mstar}). Hence $M^{*}=0$. By Lemma~\ref{lem-optval}, the optimal
value equals $C(\hat w)+M^{*}=C(\hat w)$. It remains to find a feasible solution of cost at most $C(\hat w)$. Then we apply Lemma~\ref{lem-Vprop}(3) at $M=0$. Its premise is
$V(0)\le D$, which is precisely the standing assumption of this case. The
lemma then yields a feasible solution, namely the convex combination
(\ref{eq-walpha}). It remains to check its cost. On the one hand,
Lemma~\ref{lem-Vprop}(3) gives $C(\tilde w_{\alpha})\le C(\hat w)$. On
the other hand, $\tilde w_{\alpha}$ is feasible for
$\mathrm{RIOVSPT}_{\infty}(\ddot T)$, so Lemma \ref{lem-lowbd-hard} gives
$C(\tilde w_{\alpha})\ge C(\hat w)$. Hence $C(\tilde w_{\alpha})=C(\hat
w)$. Therefore $\tilde w_{\alpha}$ is optimal.\hb
\end{proof}

\subsubsection {Attainability of a budget: the margin criterion.}\label{sss-margins}
It remains to find the smallest $M$ with $V(M)\le D$.  We certify budget attainability through truncated-bound margins. The optimal value is the smallest attainable budget.

From now on we assume $V(0)>D$, then $V(M^{*})=D$ by Lemma~\ref{lem-trivial}. It remains
to determine $M^{*}$ exactly. We test the condition $V(M)\le D$ directly:
by Lemma~\ref{lem-Vprop}(3), it holds if and only if the budget $B:=C(\hat w)+M$ is \textbf{attainable}.
To make this test operational, we express attainability as numerical
conditions on $B$. By definition, the budget $B$ is \textbf{attainable} if some
vector $\tilde w$ meets three groups of requirements of the problem \textbf{RIOVSPT$_{\infty}(\ddot T)$}: the original bounds
$l\le\tilde w\le u$, the budget $c(e)\big|\tilde w(e)-w(e)\big|\le B$ on
every edge, and the path constraints $\tilde w(P_k)\ge D(t_k\in \ddot Y\backslash\{t_0\})$ together with
$\tilde w(P_0)=D$. The first two groups concern single edges only, and
Lemma \ref{lem-budget} fuses them into one inequality
$l_B(e)\le\tilde w(e)\le u_B(e)$: the budget is not dropped but folded
into the bounds, so that a larger $B$ means a wider interval
$[l_B(e),u_B(e)]$. From then on $B$ no longer appears in the constraint
system, and attainability becomes the plain feasibility of a  \textbf{RIOVSPT$_{\infty}(\ddot T)$}
instance whose bounds are $(l_B,u_B)$. Such feasibility is characterized
by Theorem \ref{thm-fea=RIOVSPT_infty}. Translating its two conditions edge by edge yields exactly
the following \textbf{margins}.
 For a budget
$B\ge 0$, define
\begin{eqnarray}
g_{0}(B)&:=&D-l_{B}(P_{0}),\nonumber\\
g_{k}^{(1)}(B)&:=&u_{B}(P_{k})-D,\qquad\qquad\qquad\ t_{k}\in\ddot Y,\label{eq-margins}\\
g_{k}^{(2)}(B)&:=&u_{B}(P_{k}\setminus P_{0})-l_{B}(P_{0}\setminus P_{k}),\qquad t_{k}\in\ddot Y\setminus\{t_{0}\}.\nonumber
\end{eqnarray}
Denote by $\mathcal G$ the set consisting of all these margins.
\begin{thm}\label{thm-Bfea}
Let $B\ge 0$. The budget $B$ is attainable if and only if $g(B)\ge 0$ for every $g\in\mathcal{G}$.
\end{thm}
\begin{proof}
By Lemma~\ref{lem-budget}, the budget $B$ is attainable if and only if
the following truncated instance is feasible:
\textbf{RIOVSPT$_{\infty}(\ddot T)$} with the bound vectors $(l,u)$
replaced by $(l_{B},u_{B})$. 
It remains to characterize the feasibility of the truncated instance.
By Theorem~\ref{thm-fea=RIOVSPT_infty}, it is feasible if and only if
\begin{align}
l_{B}(P_{0})&\;\le\;D\;\le\;\min_{t_{k}\in\ddot Y}\,u_{B}(P_{k}),
\label{eq:feas-cond1}\\
l_{B}(P_{0})&\;\le\;\min_{t_{k}\in\ddot Y}\Bigl(l_{B}(P_{k}\cap P_{0})
+u_{B}(P_{k}\setminus P_{0})\Bigr).
\label{eq:feas-cond2}
\end{align}
Condition (\ref{eq:feas-cond1}) is exactly $g_{0}(B)\ge 0$ together with
$g_{k}^{(1)}(B)\ge 0$ for every $t_{k}\in\ddot Y$. For
(\ref{eq:feas-cond2}), separate $t_{0}$ from the other leaves. If
$t_{k}=t_{0}$, the right-hand side equals $l_{B}(P_{0})$ and the
inequality is trivial. If $t_{k}\in\ddot Y\setminus\{t_{0}\}$, then
$l_{B}(P_{0}\setminus P_{k})=l_{B}(P_{0})-l_{B}(P_{k}\cap P_{0})$, and
the inequality rewrites as
$l_{B}(P_{0}\setminus P_{k})\le u_{B}(P_{k}\setminus P_{0})$,
which is precisely $g_{k}^{(2)}(B)\ge 0$.\hb
\end{proof}

Define the breakpoint set
\begin{equation}\label{eq-Bset}
\mathcal B:=\{0\}\cup\big\{c(e)\big(w(e)-l(e)\big),\ c(e)\big(u(e)-w(e)\big)\,\big|\,e\in\ddot T\big\}. 
\end{equation}
Let
\begin{equation}\label{eq-Bstar}
B^{*}:=\min\big\{B\ge 0 \big| g(B)\ge 0 \text{~for all~ }g\in\mathcal{G}\big\}.
\end{equation}

\begin{lem}\label{lem-monotone}
Every margin in \emph{(\ref{eq-margins})} is a nondecreasing, continuous
and piecewise affine function of $B$ on $[0,+\infty)$. Each of its
breakpoints belongs to the set $\mathcal B$ in \emph{(\ref{eq-Bset})}.
Consequently, the set of attainable budgets is a closed ray
$[B^{*},+\infty)$. In particular, the minimum in \emph{(\ref{eq-Bstar})}
is attained.
\end{lem}
\begin{proof}
We show the properties of the margins. By Remark~\ref{rem:kinks}, each
$l_{B}(e)$ enters every margin with a single kink, nonincreasingly, and
each $u_{B}(e)$ nondecreasingly. Hence every margin $g\in\mathcal{G}$ is
nondecreasing, continuous and piecewise affine, and its breakpoints lie
in $\mathcal{B}$ of (\ref{eq-Bset}).

It remains to examine the structure of the attainable budgets. By
Theorem~\ref{thm-Bfea}, the set of attainable budgets is
$\{B\ge 0:\ g(B)\ge 0\ \text{for all }g\in\mathcal{G}\}.$
Each defining inequality is closed (its left side is continuous) and
upward closed (its left side is nondecreasing), so the set itself is
closed and upward closed. It remains to show it is nonempty. Let
$B_{\max}:=\max\mathcal{B}$. By the definition of $\mathcal{B}$,
$B_{\max}$ is no smaller than any kink location in
Remark~\ref{rem:kinks}, so $l_{B_{\max}}(e)=l(e)$ and
$u_{B_{\max}}(e)=u(e)$ for every edge $e$. Hence the margins at
$B_{\max}$ coincide with the two feasibility conditions of
Theorem~\ref{thm-fea=RIOVSPT_infty}, which hold by the standing
feasibility assumption. Thus $B_{\max}$ is attainable and the set is
nonempty. A nonempty, closed and upward closed subset of $[0,+\infty)$
is a closed ray $[B^{*},+\infty)$ containing its left endpoint. Hence
the minimum in (\ref{eq-Bstar}) is attained.\hb
\end{proof}

\begin{thm}\label{thm-optval}
The optimal value of the problem \textbf{RIOVSPT$\bm{_{\infty}(\ddot T)}$} equals $B^*$ in \emph{(\ref{eq-Bstar})}. Moreover, $B^*\geq C(\hat w)$.
\end{thm}
\begin{proof}
Consider any feasible solution $\tilde w$ of
\textbf{RIOVSPT$_{\infty}(\ddot T)$}. Its cost $C(\tilde w)$ is an
attainable budget, witnessed by $\tilde w$ itself. Hence
$C(\tilde w)\ge B^{*}$. Conversely, $B^{*}$ is attainable by
Lemma~\ref{lem-monotone}. So some feasible solution $\tilde w'$ costs
at most $B^{*}$. Therefore the optimal value equals $B^{*}$. Finally,
Lemma~\ref{lem-lowbd-hard} shows $C(\hat w)\le C(\tilde w')\le B^{*}$.\hb
\end{proof}

Theorem~\ref{thm-optval} locates the optimum through the attainable
budgets. Lemma~\ref{lem-optval} locates it through the value function
$V$. The two thresholds $B^{*}$ and $M^{*}$ must therefore agree up to
the shift $C(\hat w)$. The following corollary makes this precise by
rewriting the predicate $V(M)\le D$ in the language of margins.

\begin{corollary}\label{cor-Vlink}
For every $M\ge 0$, $V(M)\le D$ holds if and only if
$g\bigl(C(\hat w)+M\bigr)\ge 0$ for all $g\in\mathcal{G}$. Consequently,
$B^{*}=C(\hat w)+M^{*}$.
\end{corollary}
\begin{proof}
By Lemma~\ref{lem-Vprop}(3), $V(M)\le D$ holds if and only if
$\mathrm{RIOVSPT}_{\infty}(\ddot T)$ admits a feasible solution of cost
at most $C(\hat w)+M$. By Theorem~\ref{thm-Bfea}, this is equivalent to
the nonnegativity of all margins at $B=C(\hat w)+M$. Hence
\[
\{M\ge 0:\ V(M)\le D\}\;=\;\{B-C(\hat w):\ B\text{ attainable}\}.
\]
The right-hand set consists of nonnegative numbers: every attainable
budget is at least $B^{*}$ by the definition of $B^{*}$, and
$B^{*}\ge C(\hat w)$ by Theorem~\ref{thm-optval}. Both minima exist, and
the affine change of variables $B=C(\hat w)+M$ gives
$M^{*}=B^{*}-C(\hat w)$.\hb
\end{proof}

\subsubsection{Locating the optimal budget and constructing the solution}\label{sss-final}
By Corollary~\ref{cor-Vlink}, computing $M^{*}$ amounts to computing
$B^{*}$, which Theorem~\ref{thm-optval} identifies with the optimal
value itself. It remains to determine $B^{*}$ exactly. The attainability
test of Theorem~\ref{thm-Bfea} is monotone in $B$ by
Lemma~\ref{lem-monotone}. Hence a binary search over the sorted values
of $\mathcal{B}$ locates two consecutive breakpoints $a<b$, with $a$
unattainable and $b$ attainable. Consequently, $B^{*}\in(a,b]$. Between
two consecutive breakpoints every margin is affine. The following lemma
gives $B^{*}$ in closed form.

\begin{lem}\label{lem-affine}
Let $a<b$ be two consecutive values of $\mathcal B$ such that $a$ is
unattainable and $b$ is attainable. Then every margin $g\in\mathcal G$ is
affine on $[a,b]$, and
\begin{equation}\label{eq-closedform}
B^{*}=\max_{\substack{g\in\mathcal G\\ g(a)<0}}
\left(a+(b-a)\,\frac{-g(a)}{g(b)-g(a)}\right),
\end{equation}
where every denominator satisfies $g(b)-g(a)>0$.
\end{lem}
\begin{proof}
The breakpoints of every margin lie in $\mathcal{B}$. Since $a$ and $b$
are consecutive elements of $\mathcal{B}$, no margin has a kink in
$(a,b)$. Hence every $g\in\mathcal{G}$ is affine on $[a,b]$.

Fix $g\in\mathcal{G}$. If $g(a)\ge0$, then $g(B)\ge0$ throughout
$[a,b]$ by monotonicity (Lemma~\ref{lem-monotone}). Such a margin
imposes no constraint. If $g(a)<0$, then $g(b)\ge0$, because $b$ is
attainable (Theorem~\ref{thm-Bfea}). Thus $g$ increases strictly, and
its unique zero in $(a,b]$ is
\[
B_g:=a+(b-a)\,\frac{-g(a)}{g(b)-g(a)}.
\]
Moreover, $g(B)\ge0$ on $[a,b]$ if and only if $B\ge B_g$.

Finally, $a$ is unattainable, so the family
$\{g\in\mathcal{G}:g(a)<0\}$ is nonempty (Theorem~\ref{thm-Bfea}). Take
any $B\in[a,b]$. All margins are nonnegative at $B$ if and only if
$B\ge B_g$ for every $g$ in this family. Hence the smallest attainable
budget in $[a,b]$ equals the maximum in (\ref{eq-closedform}). No
budget below $a$ is attainable by monotonicity. And $b$ itself is
attainable. Therefore this maximum is exactly $B^{*}$.\hb
\end{proof}

The breakpoint search of Lemma~\ref{lem-affine} gives the optimal budget
$B^{*}$. The canonical solver by Algorithm \ref{alg-SPIT0} then returns the canonical solution
$\tilde w^{M^{*}}$ at this budget. The next theorem shows that this
solution is optimal for \textbf{RIOVSPT$_{\infty}(\ddot T)$}.
 \begin{thm}\label{thm-final}
Suppose $V(0)>D$. Let $B^{*}$ be the value of (\ref{eq-Bstar}) and
$M^{*}:=B^{*}-C(\hat w)$. Then the canonical solution
$\tilde w^{M^{*}}$ returned by Algorithm~\ref{alg-SPIT0} is optimal for
\textbf{RIOVSPT$_{\infty}(\ddot T)$}, and
$C(\tilde w^{M^{*}})=C(\hat w)+M^{*}=B^{*}$.
\end{thm}
\begin{proof}
Since $V(0)>D$, Lemma~\ref{lem-trivial} yields $V(M^{*})=D$, that is,
$\tilde w^{M^{*}}(P_{0})=D$. By Theorem~\ref{thm-SPITopt},
$\tilde w^{M^{*}}$ is feasible for
$\mathbf{SPIT}_{\infty}^{0}(M^{*})$, so its cost is at most
$C(\hat w)+M^{*}$ and its   paths on $\ddot T\backslash P_0$ meet the bound $D$. Both
properties carry over to \textbf{RIOVSPT$_{\infty}(\ddot T)$}. Hence $\tilde w^{M^{*}}$ is
feasible for \textbf{RIOVSPT$_{\infty}(\ddot T)$}.

Its cost satisfies $C(\tilde w^{M^{*}})\le C(\hat w)+M^{*}=B^{*}$ by Corollary~\ref{cor-Vlink}. On the other hand, every
feasible solution of \textbf{RIOVSPT$_{\infty}(\ddot T)$} costs at
least $B^{*}$ by Theorem~\ref{thm-optval}. Therefore
$C(\tilde w^{M^{*}})=B^{*}$  and $\tilde w^{M^{*}}$ is optimal.\hb
\end{proof}

The whole procedure for the hard case is summarized in the \textbf{ELSE} branch of Algorithm \ref{alg-IOVSPTinfty}, which also collects the other two cases for completeness.

\begin{breakablealgorithm}
	\renewcommand{\thealgorithm}{2}
	\caption{$[\tilde  w,C(\tilde   w)]:=\textbf{RIOVSPT}_\infty(T, Y, w, l, u, c, D,t_0)$}
	\label{alg-IOVSPTinfty}
	\begin{algorithmic}[1]

		\REQUIRE  The tree $T(V,E)$ rooted at $v_1$, the set $Y$ of leaf nodes; three weight vectors $w,l,u$ and the cost vector $c$; the given leaf node $t_0$ and $D$.
		\ENSURE  An optimal solution $\tilde w$  and the optimal value $C(\tilde w)$.
        \STATE Compute $P_0$, its first edge $\ddot e$, the gate subtree $\ddot T$, $\ddot Y$, and the   chain $\chi_{t_0}$. 
		\IF {$D<l(P_0)$ or $ D>\min_{t_k\in Y}u(P_k)$ or $l(P_0)>\min_{t_k\in Y}\big(l(P_k\cap P_0)+u(P_k\backslash P_0)\big)$}
		\RETURN ``The problem is infeasible." 
		\ENDIF
		\STATE Call $[\hat w^0, C(\hat w^0)]:=MCSPIT_{\infty}(T\backslash \ddot T, Y\backslash\ddot Y, w, u, c, D)$. %
		\STATE Call $[\hat w,C(\hat w)]:=MCSPIT_\infty(\ddot T\backslash\chi_{t_0},\ddot Y\backslash\{t_0\},w,u,c,D)$.
\STATE Set $\hat w(e):=w(e)$ on $\chi_{t_0}$. Let  $D^0:=D-\hat w(P_0)$.
         \STATE Compute $\Delta\hat l(e)$$:=\max\{l(e)-w(e),-C(\hat w)/c(e)\}$ and $\Delta\hat u(e):=\min\{u(e)-w(e),C(\hat w)/c(e)\}$ on $\chi_{t_0}$.

		\IF {$\Delta \hat l(\chi_{t_0})\leq D^0\leq \Delta \hat u(\chi_{t_0})$} %
			\IF {$D^0\leq 0$}
				\STATE Determine $\tilde w$ on $\ddot T$ by (\ref{tilde w-case-1}).
			\ELSE
				\STATE Determine $\tilde w$ on $\ddot T$ by (\ref{tilde w-case-2}).
			\ENDIF
			\STATE Let $C(\tilde w):=C(\hat w)$.

		\ELSIF {$D^0>\Delta \hat u(\chi_{t_0})$} %
			\STATE Sort the distinct values of
$\{c(e)\Delta u(e)\mid e\in P_{0},  c(e)\Delta u(e)>C(\hat w)\}$
in ascending order as %
$
\lambda_{1}:=c(e_{\alpha_{1}})\Delta u(e_{\alpha_{1}})<\cdots
<\lambda_{m}:=c(e_{\alpha_{m}})\Delta u(e_{\alpha_{m}}).
$
\\
Set $\lambda_{0}:=C(\hat w)$.
\STATE Let $a:= 0$ and $b:= m$.
\WHILE{$b-a>1$}
    \STATE $k:=\lfloor\frac{a+b}{2}\rfloor$ and compute $L(\lambda_{k})$ by (\ref{eq-Laffine}).
    \IF{$L(\lambda_{k})<D$}
        \STATE $a:= k$.
    \ELSE
        \STATE $b:=k$.
    \ENDIF
\ENDWHILE
\STATE  set $E_{<}^{*}:=\{e\in P_{0}\mid c(e)\varDelta u(e)\le\lambda_{a}\}$ and $E_{\ge}^{*}:=\{e\in P_{0}\mid c(e)\varDelta u(e)>\lambda_{a}\}$.
\STATE Compute $\lambda^{*}$ by (\ref{eq-lamstar-closed}) and determine $\tilde{w}$ on $\ddot{T}$ by (\ref{eq-wstar-up}).
\STATE Let $C(\tilde{w}):=\lambda^{*}$.
			
		\ELSE %
			\STATE Mark the nodes of $P_0$ and compute the merge node  $v^*(t_k)$  for every leaf $t_k\in\ddot Y\backslash\{t_0\}$ by one depth-first traversal. 
			\STATE Pre-adjust $\hat w$ on $\chi_{t_0}$ by (\ref{eq-free-down}).
			\STATE Call $[\tilde w^0,\tilde w^0(P_0)]:=\textbf{SPIT}_\infty^0(M)(\ddot T, \ddot Y, w,\hat w, l, u, c, D,0,t_0)$.

			\IF {$\tilde w^0(P_0)\leq D$}
				\STATE Determine $\tilde w$ on $\ddot T$ by the convex combination $\tilde w_\alpha$  as (\ref{eq-walpha})  and let $C(\tilde w):=C(\hat w)$.
			\ELSE

\STATE Sort and deduplicate the breakpoint set $\mathcal{B}$ of (\ref{eq-Bset}) as $\beta_{0}:=0<\beta_{1}<\cdots<\beta_{q}$. 

\STATE $a:=0$ and $b:=q$. 
\WHILE{$b-a>1$}
    \STATE $k:=\lfloor\frac{a+b}{2}\rfloor$. Evaluate all margins of (\ref{eq-margins}) at $B=\beta_{k}$ by one depth-first traversal.
    \IF{$g(\beta_k)\ge 0$ for all $g\in \mathcal{G}$} 
        \STATE $b:=k$.
    \ELSE
        \STATE $a:=k$.
    \ENDIF
\ENDWHILE
\STATE Set the budget values $a:=\beta_{a}$ and $b:=\beta_{b}$. Then $a$ is unattainable, $b$ is attainable, and no edge breakpoint lies in $(a,b)$.
                \STATE Evaluate the margins at these budget values $a,b$, compute $B^*$ by (\ref{eq-closedform}), and let $M^*:=B^*-C(\hat w)$.
 
				\STATE Call $[\tilde w,\tilde w(P_0)]:=\textbf{SPIT}_\infty^0(M)(\ddot T, \ddot Y, w,\hat w, l, u, c, D,M^*,t_0)$.\\
Let $C(\tilde w):=C(\hat w)+M^*$. 
			\ENDIF
		\ENDIF

		\STATE Extend $\tilde w$ to the whole tree by $\tilde w(e):=\hat w^0(e)$ for $e\in T\backslash \ddot T$. 
		\RETURN $\Big(\tilde  w,\ \max\big\{C(\hat w^0), C(\tilde w)\big\}\Big)$

	\end{algorithmic}
\end{breakablealgorithm}

It remains to analyze the running time of the overall procedure.

\begin{thm}\label{thm-hardtime}
In the hard case $\Delta\hat l(\chi_{t_0})>D^0$, the subproblem \emph{\textbf{RIOVSPT$\bm{_\infty}$}$\bm{(\ddot T)}$} is solved in $O(n\log n)$ time.
\end{thm}
\begin{proof}
(1) One call of Algorithm~\ref{alg-SPIT0}. Constructing
$\tilde w^{1}$ by (\ref{tildew1}) takes $O(n)$ time. All root--leaf
lengths $\tilde w^{1}(P_t)$ and the merge node of every leaf can be
precomputed by one depth-first traversal, again in $O(n)$ time.

The key observation is uniform shortening. At the beginning of
iteration $\tau$, every leaf of $Y^{\tau}$ contains the whole processed
subpath $P_{v_1,v^{\tau-1}}$ of $P_0$. Only edges of this subpath have
been downgraded. Hence the current length of every $t\in Y^{\tau}$
equals $\tilde w^{1}(P_t)-\xi^{\tau}$, where $\xi^{\tau}$ is the total
amount downgraded so far, the same for all active leaves. Therefore, we have 
$
t^{\tau}=\arg\min_{t\in Y^{\tau}}\tilde w^{1}(P_t), 
D^{\tau}=\min_{t\in Y^{\tau}}\tilde w^{1}(P_t)-\xi^{\tau}-D.
$
Thus $t^{\tau}$ and $D^{\tau}$ come from the heap minimum and  no leaf is
rescanned.

The active set is never rebuilt explicitly. Instead, all leaves of $\ddot Y\setminus\{t_{0}\}$ stay in a priority
queue implemented as a min-heap: its extract-min returns the leaf with
the smallest initial length.  The key of
leaf $t$ is its initial length $\tilde w^{1}(P_t)$, and ties are
broken by a fixed rule.
Keep the depth of the processed endpoint. Before each extraction,
discard the heap root while its merge node is no deeper than that
endpoint; such a leaf is inactive. Every leaf is inserted once and
physically removed at most once. All heap operations cost $O(n\log n)$
in total. One iteration may discard several leaves, so no
$O(\log n)$ worst-case cost per iteration is claimed.

The remaining costs are linear. Precomputing prefix sums on $P_0$ gives each
$R\bigl(P_{v^{\tau-1},v^{\tau}}\bigr)$ in $O(1)$ time: it is the
difference of two prefix values. The pointer locating the partial edge moves only
downward. Each iteration removes at least the selected leaf, so there
are at most $|\ddot Y|-1$ iterations. The notation $\tilde w^{\tau}$ is analytical: the algorithm updates
one shared array in place. Since the processed prefix moves only
toward $t_0$, the subpaths of different iterations are disjoint, and
each edge is modified at most once. A single iteration may touch
$O(n)$ edges, but the total edge work is $O(|P_0|)$. A final depth-first traversal recomputes the root--leaf
lengths and classifies every side leaf in $O(n)$ time. Consequently,
one call takes $O(n\log n)$ time.

(2) The search for $M^{*}$. First, mark the nodes of $P_{0}$
and run one depth-first traversal of $\ddot T$. It records, for every
node $v$, its merge node $v^{*}(v)$, the deepest node of $P_{0}$ on
$P_{v_1,v}$, and the depth $d_0(v)$. This preprocessing takes $O(n)$
time and is done once. For a candidate budget $B$, one further
depth-first traversal accumulates the prefix sums
$\ell_B(v):=l_B(P_{v_1,v})$ and $\mu_B(v):=u_B(P_{v_1,v})$. Then the
truncated bounds (\ref{eq-trunc}) of all edges and all the margins
(\ref{eq-margins}) are obtained in $O(n)$ time: each margin is a
difference of two such values, since
$l_B(P_0)=\ell_B(t_0)$,
$u_B(P_k\setminus P_0)=\mu_B(t_k)-\mu_B(v^{*}(t_k))$, and
$l_B(P_0\setminus P_k)=\ell_B(t_0)-\ell_B(v^{*}(t_k))$. Hence one
attainability test of Theorem~\ref{thm-Bfea} costs $O(n)$ time.

The breakpoint set $\mathcal{B}$ of (\ref{eq-Bset}) has at most
$2|E(\ddot T)|+1=O(n)$ elements. It is constructed and sorted in
$O(n\log n)$ time. The binary search performs $O(\log n)$
attainability tests, costing $O(n\log n)$ time in total. The closed
form (\ref{eq-closedform}) is evaluated from the margins at two
consecutive breakpoints in $O(n)$ time. Algorithm~\ref{alg-SPIT0} is
called twice, at $M=0$ and at $M=M^{*}$, costing $O(n\log n)$ time by
part (1). Altogether, the hard case costs $O(n\log n)$ time.\hb
\end{proof}

We conclude this subsection by summarizing the three cases.
\begin{thm}\label{thm-RIOVSPTinfty}
The problem \emph{\textbf{(RIOVSPT$\bm{_\infty}$)}} on a tree with $n\geq2$ nodes can be solved in $O(n\log n)$ arithmetic operations.
\end{thm}
\begin{proof}
By Theorem~\ref{thm-construct}, the problem
\textbf{RIOVSPT$\bm{_\infty}$}$\bm{(T)}$ decomposes into two
independent subproblems. The subproblem on $T\backslash\ddot T$ is
solved by one call of \textbf{MCSPIT$\bm{_\infty}$} at the same target
$D$ (Theorem~\ref{thm-construct}). It remains to bound the subproblem
\textbf{RIOVSPT$\bm{_\infty}$}$\bm{(\ddot T)}$.

Three cases are distinguished. The middle case
$\Delta\hat l(\chi_{t_0})\leq D^0\leq\Delta\hat u(\chi_{t_0})$ and the
upgrading case $\Delta\hat u(\chi_{t_0})<D^0$ are solved in closed form
(Theorems~\ref{thm-case-mid} and~\ref{thm-opt-up}), each in $O(n\log n)$
time. The hard case $\Delta\hat l(\chi_{t_0})>D^0$ takes $O(n\log n)$
time by Theorem~\ref{thm-hardtime}. Algorithm~\ref{alg-IOVSPTinfty}
combines the three cases.

The total cost is dominated by three parts. The feasibility test costs
$O(n)$ time by prefix sums. The two calls of
\textbf{MCSPIT$\bm{_\infty}$} cost $O(n\log n)$ arithmetic operations.
Among the three cases, the middle one is a closed-form evaluation in
$O(n)$ time, the upgrading one sorts $O(n)$ values and binary-searches
with $O(1)$-time evaluations, and the final extension costs $O(n)$
time. Altogether, the problem is solved in $O(n\log n)$ arithmetic
operations.\hb
\end{proof}

\subsection{A direct budget construction valid in all cases}\label{subsec-uniform}
The characterization in Theorem~\ref{thm-Bfea} depends only on the data
$(\ddot T,w,l,u,c,D,t_{0})$: it uses neither the initial vector
$\hat w$ nor the case division. Together with the closed form
(\ref{eq-closedform}), it determines the optimal value $B^{*}$ in
$O(n\log n)$ time in every case. We nevertheless kept the three-case
development of the previous subsections, because those cases yield
explicit, closed-form optimal solutions and display the dependence on
the problem \textbf{$MCSPIT_{\infty}$}: the optimal value stays
$C(\hat w)$ in the first case and equals the critical value
$\lambda^{*}$ in the second, while in the hard case it equals
$B^{*}=C(\hat w)+M^{*}\ge C(\hat w)$, with an optimal solution produced
by Algorithm~\ref{alg-SPIT0}. In this section we push the same
observation further: applied to the full tree $T$, the margin test
yields a single algorithm that covers all cases at once.

We first use the full tree $T$ and all leaves $Y$ in (\ref{eq-trunc}) and (\ref{eq-margins}), and denote the resulting margin family by $\mathcal G_T$. Define
\[
\mathcal B_T=\{0\}\cup\{c(e)(w(e)-l(e)),\ c(e)(u(e)-w(e)):e\in E\}.
\]
Test the original feasibility conditions of Theorem~\ref{thm-fea=RIOVSPT_infty} first. For a feasible instance, if all margins are nonnegative at zero then $w$ itself is optimal. Otherwise zero is unattainable and $\max\mathcal B_T$ is attainable. Thus there are consecutive distinct breakpoints $a<b$ with $a$ unattainable and $b$ attainable, and Lemma~\ref{lem-affine} gives the full-tree optimum
\[
B_T^*=\max_{g\in\mathcal G_T:\,g(a)<0}
\left(a+(b-a)\frac{-g(a)}{g(b)-g(a)}\right).
\]
Once $B_{T}^{*}$ is computed, applying the constructive part of
Theorem~\ref{thm-fea=RIOVSPT_infty} with $(l_{B_{T}^{*}},u_{B_{T}^{*}})$
in place of $(l,u)$ yields an optimal vector.
Algorithm~\ref{alg-uniform} below implements this construction directly.

\begin{breakablealgorithm}
\renewcommand{\thealgorithm}{3}
\caption{Direct solver.$[\tilde  w,C(\tilde   w)]:=\textbf{RIOVSPT}_\infty(T, Y, w, l, u, c, D,t_0)$}
\label{alg-uniform}
\begin{algorithmic}[1]
\REQUIRE  The tree $T(V,E)$ rooted at $v_1$, the set $Y$ of leaf nodes; three weight vectors $w,l,u$ and the cost vector $c$; the given leaf node $t_0$ and $D$.
		\ENSURE  An optimal solution $\tilde w$  and the optimal value $C(\tilde w)$.
\STATE Mark $P_0$, compute the merge node $v^{*}(t_{k})$ of every
$t_{k}\in Y$.
\IF {$D<l(P_0)$ or $ D>\min_{t_k\in Y}u(P_k)$ or $l(P_0)>\min_{t_k\in Y}\big(l(P_k\cap P_0)+u(P_k\backslash P_0)\big)$}
		\RETURN ``The problem is infeasible." 
		\ENDIF
\STATE Evaluate every margin in $\mathcal G_T$ at $B=0$.
\IF{all these margins are nonnegative}
    \RETURN $(w,0)$.
\ENDIF
\STATE Sort and deduplicate the breakpoint set $\mathcal{B}_{T}$ as
$\beta_{0}:=0<\beta_{1}<\cdots<\beta_{q}$. 
\STATE Let $a:=0$ and $b:=q$. 
\WHILE{$b-a>1$}
    \STATE $k:=\lfloor(a+b)/2\rfloor$ and evaluate all margins
in $\mathcal G_T$ at $B=\beta_k$ by a depth-first traversal. 
    \IF{all these margins are nonnegative}
        \STATE $b:=k$.
    \ELSE
        \STATE $a:=k$.
    \ENDIF
\ENDWHILE
\STATE Evaluate all margins at $\beta_{a}$ and $\beta_{b}$, and compute
$B_{T}^{*}$ by the closed form (\ref{eq-closedform}). 
\STATE Set $x(e):=u_{B_{T}^{*}}(e)$ for $e\notin P_{0}$, and
$x(e):=l_{B_{T}^{*}}(e)$ for $e\in P_{0}$.
\STATE $\Delta:=D-l_{B_{T}^{*}}(P_{0})$.
\FOR{each edge $e\in P_{0}$ from the root to $t_{0}$}
    \STATE $\delta:=\min\bigl\{\Delta,\,u_{B_{T}^{*}}(e)-l_{B_{T}^{*}}(e)\bigr\}$;
    set $x(e):=l_{B_{T}^{*}}(e)+\delta$ and $\Delta:=\Delta-\delta$.
\ENDFOR
\RETURN $(x,\,B_{T}^{*})$.

\end{algorithmic}
\end{breakablealgorithm}
\begin{lem}\label{lem-construct}
Let $B\ge 0$ be such that all margins in $\mathcal{G}_{T}$ are
nonnegative. Then the vector $x$ produced by Lines 19--24 of
Algorithm~\ref{alg-uniform} with $B_{T}^{*}$ replaced by $B$ satisfies
\begin{equation}\label{eq-construct-range}
l_{B}(e)\le x(e)\le u_{B}(e)\quad(e\in E),
\end{equation}
$x(P_{0})=D$, and $x(P_{k})\ge D$ for every $t_{k}\in Y\setminus\{t_{0}\}$.
\end{lem}

\begin{proof}
Write $\Delta:=D-l_{B}(P_{0})$ and $L:=l_{B}$, $U:=u_{B}$ for short.
Nonnegativity of $g_{0}$ gives $\Delta\ge 0$, and nonnegativity of
$g_{t_{0}}^{(1)}$ gives $\Delta\le U(P_{0})-L(P_{0})$. Hence the filling
along $P_0$ never gets stuck: each edge receives at most its width
$U(e)-L(e)$, and the loop ends with $\Delta$ exactly exhausted, i.e.,
$x(P_{0})=D$. Also (\ref{eq-construct-range}) holds by construction.

It remains to check the side paths. Fix $t_{k}\in Y\setminus\{t_{0}\}$
and let $v^{*}$ be its merge node. Since edges outside $P_{0}$ are set
to their upper bounds,
$x(P_{k})=x(P_{v_{1},v^{*}})+U(P_{k}\setminus P_{0})$.
Put $Q=P_{v_1,v^*}$ and let $\Delta_0=D-L(P_0)$ be the initial
amount to be filled. Root-first filling gives the prefix identity
\[
x(Q)=L(Q)+\min\{\Delta_0,U(Q)-L(Q)\}.
\]
If $\Delta_0\ge U(Q)-L(Q)$, then $x(Q)=U(Q)$ and
$g_k^{(1)}(B)\ge0$ gives $x(P_k)\ge D$.
If $\Delta_0\le U(Q)-L(Q)$, then
\[
x(Q)=L(Q)+\Delta_0=D-L(P_0\setminus P_k)
\ge D-U(P_k\setminus P_0)
\]
by $g_k^{(2)}(B)\ge0$. Thus the side-path constraint holds in both
cases, including equality, zero-width edges, and zero total filling;
no partially filled edge needs to exist.\hb
\end{proof}
\begin{thm}\label{thm-uniform}
Algorithm~\ref{alg-uniform} solves $\mathrm{RIOVSPT}_{\infty}$ exactly in 
  $O(n\log n)$ time. 
\end{thm}
\begin{proof}
We first verify correctness. Theorem~\ref{thm-Bfea} applies verbatim to
the full tree, since its proof uses only
Theorem~\ref{thm-fea=RIOVSPT_infty}. Hence the infeasibility exit and
the zero-budget exit are valid, and for every remaining instance the
computed $B_{T}^{*}$ is the smallest attainable budget
(Lemmas~\ref{lem-monotone} and~\ref{lem-affine}, full-tree versions).
By Lemma~\ref{lem-construct} with $B=B_{T}^{*}$, the returned vector
$x$ respects the truncated bounds, meets every path constraint, and
satisfies $x(P_{0})=D$; its cost is therefore at most $B_{T}^{*}$.
Optimality follows: any smaller cost would be an attainable budget
below $B_{T}^{*}$, contradicting the minimality of $B_{T}^{*}$.

For the running time, sorting $O(n)$ breakpoints and running $O(\log n)$
linear-time margin tests take $O(n\log n)$ operations altogether, and
the endpoint evaluation and the final construction are both $O(n)$.
Hence the total is $O(n\log n)$.\hb
\end{proof}

 \section{Computational experiments}\label{sec:computational-study}
We implemented Algorithms~1--3 and tested them against two
general-purpose linear programming (LP) solvers, the open-source solver
HiGHS and the commercial solver Gurobi. The experiments address two
questions: whether the algorithms return the correct optimal value, and
how their running time compares with that of the LP solvers. The code,
the instance generators and all raw results are publicly available at
\url{https://github.com/xiaoli2576/RIOVSPT-infty-code}.

\subsection{LP baseline and test instances}\label{subsec:instances}
\smallskip\noindent\emph{LP baseline.} Since $\tilde w(P_k)$ is the distance from the
root $v_1$ to the leaf $t_k$, the problem \textbf{RIOVSPT$\bm{_\infty}(T)$}
in (\ref{eq-RIOVSPTinfty}) can be written as an LP with a distance
variable $\pi(v)$ for every node $v$. For each edge $e=(p_e,v_e)$, where
$p_e$ is the end closer to the root,
\begin{eqnarray}
&\min& z \nonumber\\
&s.t.& \pi(v_1)=0,\quad \pi(t_0)=D,\quad \pi(t_k)\ge D,\ t_k\in Y\setminus\{t_0\}, \nonumber\\
&& \pi(v_e)-\pi(p_e)=\tilde w(e),\quad l(e)\le\tilde w(e)\le u(e),\ e\in E, \label{eq-epiLP}\\
&& -z\le c(e)\big(\tilde w(e)-w(e)\big)\le z,\ e\in E. \nonumber
\end{eqnarray}
The LP has $O(n)$ variables and constraints and uses none of the results
of Section~2, so it provides an independent check. We accept the output of
an algorithm when both the algorithm and the LP solvers report
infeasibility, or when its solution satisfies all constraints of
(\ref{eq-RIOVSPTinfty}) and its objective value agrees with the LP optimum,
both up to $2\times10^{-6}$.

\smallskip\noindent\emph{Test trees.} We use five tree families
(Fig.~\ref{fig:tree-families}). A \emph{chain} is a single path. In a
\emph{star}, every leaf is a child of the root. A \emph{balanced} tree is a
complete binary tree. A \emph{comb} consists of a long path, the spine,
with one leaf attached to each spine node. A \emph{random} tree is grown by
attaching each new node to an existing node chosen uniformly at random. The
families range from shallow (star, balanced) to deep (chain, comb) and from
one leaf (chain) to $n-1$ leaves (star).

\begin{figure}[t]\centering
\includegraphics[width=\linewidth]{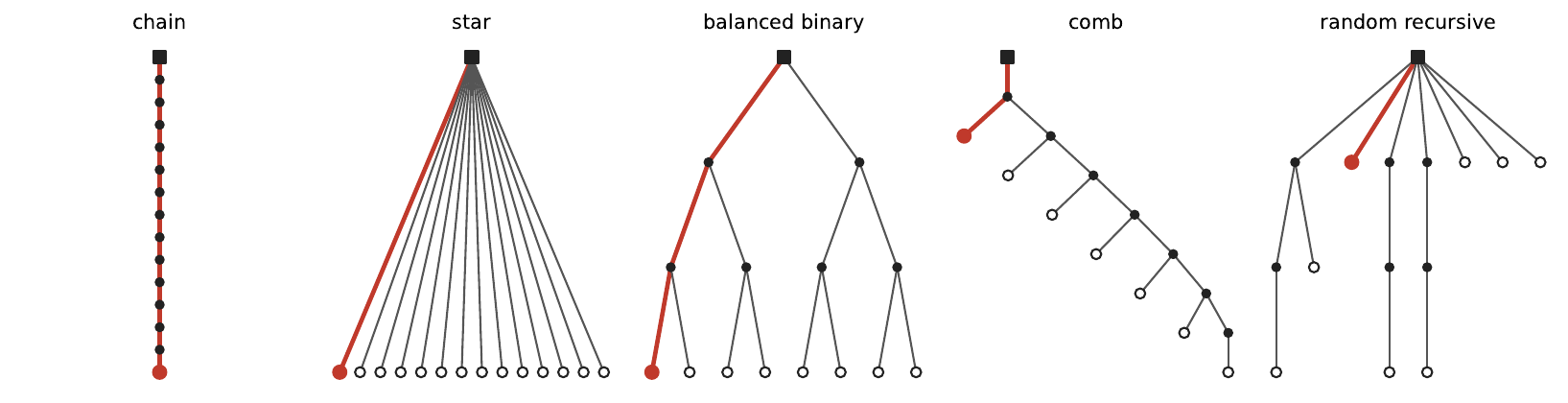}
\caption{The five tree families, shown with 14 edges. The square is the
root, open circles are leaves, and the thick path leads to the target leaf
$t_0$ used in the timing experiments.}\label{fig:tree-families}
\end{figure}

\smallskip\noindent\emph{Edge data.} On each tree, the weights $w$, the bounds $l,u$, the
costs $c$, the target leaf $t_0$ and the length $D$ are drawn at random.
The \emph{correctness} instances are not screened, so both feasible and
infeasible instances occur. The \emph{timing} instances are drawn so that
the problem is feasible, with $t_0$ the shallowest leaf and $D$ close to
$w(P_0)$; both LP solvers confirmed feasibility. The generators and all
instances are included with the code.

\subsection{Correctness}
The correctness set contains 400 random instances (the five families with
1, 5, 20 and 100 edges, 20 instances each) and 13 hand-made instances
covering corner cases such as a single-node tree, edges with $l(e)=u(e)$,
and every branch of Algorithm~2. Of these 413 instances, 90 are feasible
and 323 are infeasible. Algorithms~2 and~3 agreed with both LP solvers on
all of them. They also agreed on all 186 timing instances of the next
subsections (on the 15 largest ones, only Gurobi was used). Over all 599
instances, the largest difference between an algorithm's optimal value and
the LP optimum was $3.1\times10^{-13}$. Algorithm~1 was checked separately
on 24 trees, each at four budgets $M$; its value agreed with the
corresponding LP in all 96 cases.

\subsection{Running time}
The main timing set has 75 instances: the five families with 100, 1,000
and 10,000 edges and five instances each. A further 15 instances have
100,000 edges. All methods use one thread on an Intel Core i9-14900HX, and
each instance is solved three times and the median time is reported. The
times of the LP solvers include building the model.
Table~\ref{tab:times} and Fig.~\ref{fig:runtime} summarize the results.

\begin{table}[t]
\centering\small
\caption{Running time in seconds (median over instances). HiGHS was not run
on the 100,000-edge trees, and Gurobi was run once per instance there.}
\label{tab:times}
\begin{tabular}{lrrrrr}\hline
Family & Edges & Algorithm 2 & Algorithm 3 & HiGHS & Gurobi \\\hline
chain    & 10,000  & 0.100 & 0.093 & 5.37  & 0.714 \\
star     & 10,000  & 0.067 & 0.121 & 0.313 & 0.299 \\
balanced & 10,000  & 0.119 & 0.109 & 0.617 & 0.477 \\
comb     & 10,000  & 0.115 & 0.104 & 6.34  & 0.704 \\
random   & 10,000  & 0.030 & 0.105 & 0.451 & 0.386 \\\hline
chain    & 100,000 & 0.834 & 0.813 & --    & 43.6  \\
star     & 100,000 & 0.477 & 1.01  & --    & 2.15  \\
balanced & 100,000 & 0.881 & 0.880 & --    & 4.59  \\
comb     & 100,000 & 0.978 & 0.851 & --    & 284   \\
random   & 100,000 & 0.460 & 0.917 & --    & 6.17  \\
\hline\end{tabular}
\end{table}

\begin{figure}[t]\centering
\includegraphics[width=\linewidth]{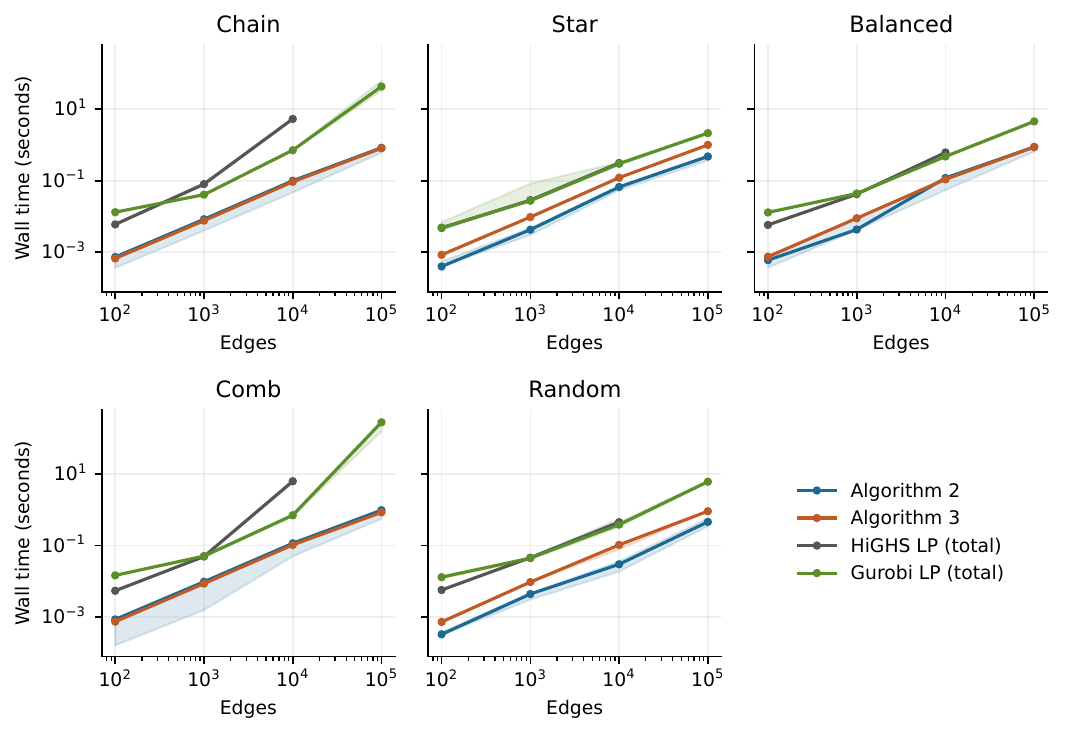}
\caption{Running time against the number of edges (log--log scale). Lines
show medians and shaded bands the interquartile range over
instances.}\label{fig:runtime}
\end{figure}

Both algorithms were faster than both LP solvers on all 90 instances. On
the 75 main instances, Gurobi took a median of 10.6 times as long as
Algorithm~2 and 5.9 times as long as Algorithm~3; HiGHS was similar. The gap
is largest on deep trees: with 100,000 edges, Gurobi needed 43.6 seconds on
chains and 284 seconds on combs, whereas neither algorithm took more than
1.2 seconds on any tree. The running times of both algorithms grow almost
linearly with $n$ on every family (Fig.~\ref{fig:runtime}), as expected from
the $O(n\log n)$ bound. Part of the gap comes from building the LP model; if
only Gurobi's solve time is counted, Gurobi was faster than Algorithm~3 on
11 star trees with at least 1,000 edges and faster than Algorithm~2 on one
star tree with 100,000 edges. 

Algorithm~2 was faster than Algorithm~3 on 60 of the 90 instances. The
next experiment shows that the difference depends on the branch of
Algorithm~2 that the instance enters.

\subsection{Sensitivity and branch-specific behavior}
\smallskip\noindent\emph{Sensitivity.} On random trees with 1,000 edges, we varied one
factor at a time: the spread of the costs $c$, the width of the bounds
$[l,u]$, and the depth of the target leaf $t_0$ (60 instances in total,
Fig.~\ref{fig:sensitivity}). None of these factors changed the running
time much. Only equal costs made both algorithms somewhat faster. Gurobi
remained between 2.9 and 12 times slower than either algorithm.

\begin{figure}[t]\centering
\includegraphics[width=\linewidth]{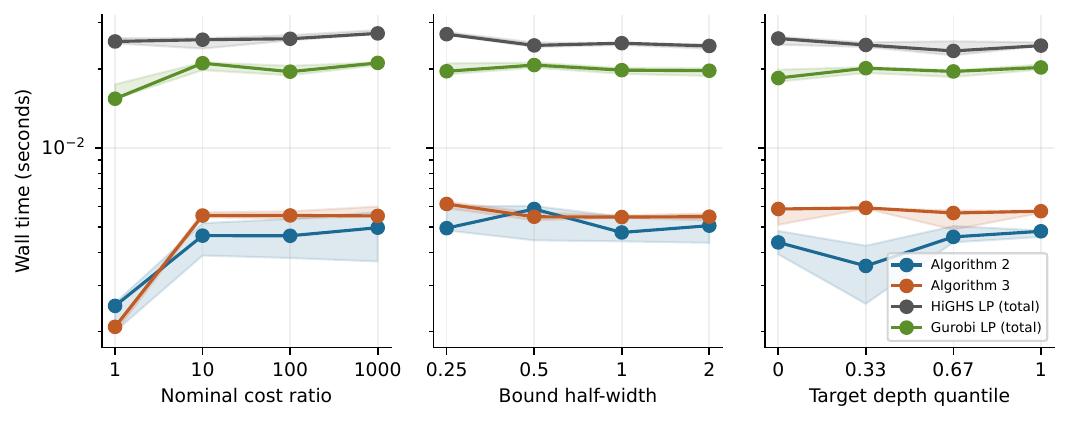}
\caption{Running time on random trees with 1,000 edges when one factor is
varied. Lines show medians and shaded bands the interquartile range over
five instances.}\label{fig:sensitivity}
\end{figure}

\smallskip\noindent\emph{Branches of Algorithm~2.} We built nine instances (three sizes
up to 10,000 edges) for each branch of Algorithm~2, using a small template
that forces the branch. Every instance entered the intended branch, and
both algorithms agreed with the LP solvers. Table~\ref{tab:branches}
compares the two algorithms. Algorithm~2 was faster in the middle, upward
and hard-zero cases, where it stops before the parametric search. In the
hard case with $V(0)>D$, Algorithm~2 also needs a budget search, in
addition to its own preprocessing, and Algorithm~3 was faster on all nine
instances.

\begin{table}[t]
\centering\small
\caption{Algorithm~2 versus Algorithm~3 on instances that force each branch
of Algorithm~2 (nine instances per branch); $t_2$ and $t_3$ are the
running times of Algorithms~2 and~3.}\label{tab:branches}
\begin{tabular}{lcc}\hline
Branch & Faster & Median $t_3/t_2$ \\\hline
middle ($\Delta\hat l(\chi_{t_0})\le D^0\le\Delta\hat u(\chi_{t_0})$) & Alg.~2 on 9/9 & 1.92 \\
upward ($\Delta\hat u(\chi_{t_0})<D^0$) & Alg.~2 on 9/9 & 1.88 \\
hard, $V(0)\le D$ & Alg.~2 on 9/9 & 1.62 \\
hard, $V(0)>D$ & Alg.~3 on 9/9 & 0.87 \\
\hline\end{tabular}
\end{table}

\section{Conclusion}\label{sec:conclusion}
We studied the restricted inverse optimal value problem on shortest
root--leaf paths under the weighted $l_\infty$ norm on trees. The prescribed
path must have length $D$ and remain a shortest root--leaf path, while the
maximum weighted edge adjustment is minimized. The feasibility conditions
lead to two $O(n\log n)$ algorithms. The first uses a decomposition and a
three-case analysis relative to an auxiliary interdiction solution; its
hard case is solved through a parametric greedy subproblem. The second
applies the budget-dependent feasibility margins directly to the full
tree. Their continuity, monotonicity, and affine behavior between
edge-bound breakpoints allow the optimal budget to be found by binary
search followed by an exact threshold calculation.

In the experiments, both algorithms returned the same optimal value as
the LP solvers HiGHS and Gurobi on all 599 test instances, and they were
faster than both solvers on all timing instances, typically by a factor
of 5 to 10 and by more than an order of magnitude on large deep trees. Neither algorithm dominates the other:
Algorithm~2 is faster when the instance is settled before the hard case,
and Algorithm~3 is faster in the hard case with $V(0)>D$.

Several questions remain open. First, this work is confined to trees. It would be interesting to find wider graph classes that still admit polynomial-time algorithms, such as series-parallel graphs or bounded-treewidth networks. Second, both of our algorithms run in $O(n\log n)$ time. Whether a linear-time algorithm exists remains open.  Third, the model admits several variants. For example, one may relax the exact length requirement on $P_0$ to an interval, prescribe several paths at once, or consider other underlying problems such as minimum spanning trees.

\section*{Statements and Declarations}

\noindent\textbf{Funding.} The work of Q.~Zhang was supported by the
National Natural Science Foundation of China (Grant No.~1230012046). The
work of X.~Guan was supported by the National Natural Science Foundation
of China (Grant No.~62573120). The work of P.M.~Pardalos was conducted
within the framework of the Basic Research Program at the National
Research University Higher School of Economics (HSE).

\smallskip\noindent\textbf{Competing interests.} The authors have no
competing interests to declare that are relevant to the content of this
article.

\smallskip\noindent\textbf{Data and code availability.} All test
instances are generated deterministically by the code accompanying this
paper. The implementation of Algorithms~1--3, the LP baseline, the
instance generators and the raw results of all experiments are publicly
available at \url{https://github.com/xiaoli2576/RIOVSPT-infty-code}.


\begin{thebibliography}{00}

\bibitem{BT-MP-ISP-1992}
Burton, D., Toint, P.L.: On an instance of the inverse shortest paths problem. Mathematical Programming \textbf{53}(1--3), 45--61 (1992)

\bibitem{BT-MP-ISP-1994}
Burton, D., Toint, P.L.: On the use of an inverse shortest paths algorithm for recovering linearly correlated costs. Mathematical Programming \textbf{63}(1), 1--22 (1994)

\bibitem{ZM-ZFOR-ISP-1995}
Zhang, J., Ma, Z., Yang, C.: A column generation method for inverse shortest path problems. ZOR -- Mathematical Methods of Operations Research \textbf{41}(3), 347--358 (1995)

\bibitem{XZ-JJIAM-WSP}
Xu, S., Zhang, J.: An inverse problem of the weighted shortest path problem. Japan Journal of Industrial and Applied Mathematics \textbf{12}(1), 47--59 (1995)

\bibitem{ZW-JHU-ISPH}
Zhang, B., Wang, Q.: Inverse shortest path problems under Hamming distance. Journal of Hohai University (Natural Sciences) \textbf{36}(4), 571--574 (2008)

\bibitem{CH-N-ISPL}
Cui, T., Hochbaum, D.S.: Complexity of some inverse shortest path lengths problems. Networks \textbf{56}(1), 20--29 (2010)

\bibitem{AG-MP-IOVP}
Ahmed, S., Guan, Y.: The inverse optimal value problem. Mathematical Programming \textbf{102}(1), 91--110 (2005)

\bibitem{Lv-JCAM-IOVP}
Lv, Y., Hu, T., Wan, Z.: A penalty function method for solving inverse optimal value problem. Journal of Computational and Applied Mathematics \textbf{220}(1--2), 175--180 (2008)

\bibitem{ZhangQ22}
Zhang, Q., Guan, X., Jia, J., Qian, X., Pardalos, P.M.: The restricted inverse optimal value problem on shortest path under $l_1$ norm on trees. Journal of Global Optimization \textbf{86}(1), 251--284 (2023)


\bibitem{ZGP-LBIOVMST-JGO}
Zhang, B., Guan, X., Pardalos, P.M., Wang, H., Zhang, Q., Liu, Y., Chen, S.: The lower bounded inverse optimal value problem on minimum spanning tree under unit $l_{\infty}$ norm. Journal of Global Optimization \textbf{79}(3), 757--777 (2021)

\bibitem{ZGZ-IOVMST-OL}
Zhang, B., Guan, X., Zhang, Q.: Inverse optimal value problem on minimum spanning tree under unit $l_{\infty}$ norm. Optimization Letters \textbf{14}(8), 2301--2322 (2020)

\bibitem{WGZ-CIOVMST-JOCO}
Wang, H., Guan, X., Zhang, Q., Zhang, B.: Capacitated inverse optimal value problem on minimum spanning tree under bottleneck Hamming distance. Journal of Combinatorial Optimization \textbf{41}(4), 861--887 (2021)

\bibitem{Jia2023}
Jia, J., Guan, X., Wang, H., Zhang, B., Pardalos, P.M.: Combinatorial algorithms for solving the restricted bounded inverse optimal value problem on minimum spanning tree under weighted $l_\infty$ norm. Journal of Computational and Applied Mathematics \textbf{419}, 114754 (2023)

\bibitem{Jia2024}
Jia, J., Guan, X., Zhang, B., Qian, X., Pardalos, P.M.: Combinatorial algorithms for restricted inverse optimal value problems on minimum spanning tree under weighted $l_1$ norm. Journal of Computational and Applied Mathematics \textbf{451}, 116110 (2024)

\bibitem{Jia2026}
Jia, J., Guan, X., Pardalos, P.M.: Restricted inverse optimal value problem on linear programming under weighted $l_1$ norm. Journal of Computational and Applied Mathematics \textbf{486}, 117687 (2026)

\bibitem{PT-ANV-BT}
Paleologo, G., Takriti, S.: Bandwidth trading: a new market looking for help from the OR community. AIRO News \textbf{VI}(3), 1--4 (2001)

\bibitem{ZhangQ20}
Zhang, Q., Guan, X., Pardalos, P.M.: Maximum shortest path interdiction problem by upgrading edges on trees under weighted $l_1$ norm. Journal of Global Optimization \textbf{79}(4), 959--987 (2021)

\bibitem{ZhangQ21}
Zhang, Q., Guan, X., Wang, H., Pardalos, P.M.: Maximum shortest path interdiction problem by upgrading edges on trees under Hamming distance. Optimization Letters \textbf{15}(8), 2661--2680 (2021)

\bibitem{YiL22}
Lei, Y., Shao, H., Wu, T., Liu, P.: An accelerating algorithm for maximum shortest path interdiction problem by upgrading edges on trees under unit Hamming distance. Optimization Letters \textbf{17}(2), 453--469 (2023)

\bibitem{ZhangL2025}
Zhang, Q., Li, X.: Minimum-cost shortest-path interdiction problem involving upgrading edges on trees with weighted $l_\infty$ norm. Mathematics \textbf{13}(19), 3219 (2025)

\bibitem{ZGJQ22}
Zhang, Q., Guan, X., Jia, J., Qian, X.: The sum of root-leaf distance interdiction problem by upgrading edges/nodes on trees. Journal of Combinatorial Optimization \textbf{44}(1), 74--93 (2022)

\bibitem{Li2024}
Li, X., Guan, X., Zhang, Q., Yin, X., Pardalos, P.M.: The sum of root-leaf distance interdiction problem with cardinality constraint by upgrading edges on trees. Journal of Combinatorial Optimization \textbf{48}(5), 39 (2024)

\bibitem{Li2025}
Li, X., Guan, X., Jia, J., Pardalos, P.M.: Double interdiction problem on trees on the sum of root-leaf distances by upgrading edges. Journal of Global Optimization \textbf{92}(4), 951--972 (2025)

\bibitem{Xiao2025}
Guan, X., Pardalos, P.M., Zhang, B.: Sum of root-leaf distance interdiction problems on trees. In: Inverse Combinatorial Optimization Problems. Springer Optimization and Its Applications, vol. 225, pp. 171--201. Springer, Cham (2025). \url{https://doi.org/10.1007/978-3-031-91175-0_7}


\end{thebibliography}
    \end{document}